\documentclass{article}

\usepackage[latin1]{inputenc}
\usepackage[dvips]{graphics}
\usepackage{a4wide}
\usepackage[]{latexsym}

\usepackage{multicol}
\usepackage[dvips]{color}

\usepackage{amsmath, amsthm, amsfonts}
\usepackage{verbatim}

\usepackage{amssymb}
\usepackage{amscd}
\usepackage{longtable}
\usepackage{array}
\usepackage[all]{xy}
\usepackage{a4}

\newtheorem{defn0}{Definition}[section]
\newtheorem{prop0}[defn0]{Proposition}
\newtheorem{thm0}[defn0]{Theorem}
\newtheorem{lemma0}[defn0]{Lemma}
\newtheorem{claim0}[defn0]{Claim}
\newtheorem{corollary0}[defn0]{Corollary}
\newtheorem{example0}[defn0]{Example}
\newtheorem{remark0}[defn0]{Remark}
\newtheorem{assumption0}[defn0]{Assumption}
\newtheorem{conjecture0}[defn0]{Conjecture}
\newtheorem{notation0}[defn0]{Notation}
\newtheorem{question0}[defn0]{Question}

\newenvironment{definition}{\begin{defn0}\rm}{\end{defn0}}
\newenvironment{proposition}{\begin{prop0}}{\end{prop0}}
\newenvironment{theorem}{\begin{thm0}}{\end{thm0}}
\newenvironment{lemma}{\begin{lemma0}}{\end{lemma0}}

\newenvironment{remark}{\begin{remark0}\rm}{\end{remark0}}

\newcommand{\Gal}{{\mathrm {Gal}}}

\newcommand{\disc}{{\mathrm {disc }}}
\newcommand{\ord}{\mathrm{ord}}
\newcommand{\Pic}{\mathrm{Pic}}
\newcommand{\Norm}{\mathrm{N}}
\newcommand{\E}{\mathcal{E}}
\newcommand{\Trace}{{\mathrm{Tr}}}

\newcommand{\End}{{\mathrm{End}}}

\newcommand{\Z}{{\mathbb Z}}
\newcommand{\Q}{{\mathbb Q}}
\newcommand{\C}{{\mathbb C}}
\newcommand{\R}{{\mathbb R}}
\newcommand{\F}{{\mathbb F}}

\newcommand{\PP}{{\mathbb P}}

\newcommand{\cS}{{\mathcal S}}

\newcommand{\cX}{{\mathcal X}}
\newcommand{\cY}{{\mathcal Y}}

\newcommand{\cM}{{\mathcal M}}

\newcommand{\cO}{{\mathcal O}}
\newcommand{\cW}{{\mathcal W}}

\newcommand{\dP}{{\mathfrak P}}

\newcommand{\Res}{{\mathrm {Res}}}

\newcommand{\Spec}{{\mathrm {Spec}}}

\newcommand{\ra}{{\rightarrow}}

\newcommand{\Hom}{{\mathrm {Hom}}}
\newcommand{\WP}{{\mathrm {WP}}}

\newcommand{\car}{{\mathrm{char}}}

\newcommand{\CM}{\operatorname{CM}}

\title{Equations of hyperelliptic \\ Shimura curves}
\author{Santiago Molina\\
    \small Matemàtica Aplicada IV\\
  \small Universitat Politècnica de Catalunya
}
\begin{document}
\maketitle

\begin{abstract}
We describe an algorithm that computes explicit models of hyperelliptic Shimura
curves attached to an indefinite quaternion algebra over $\Q$ and Atkin-Lehner quotients of them.
It exploits Cerednik-Drinfeld's non-archimedean uniformisation of Shimura curves, a formula of Gross and Zagier for the endomorphism ring of Heegner points over Artinian rings and the connection between Ribet's bimodules and the specialization of Heegner points, as introduced in \cite{Mol}.
As an application, we provide a list of equations of Shimura curves and quotients of them obtained by our algorithm that had been conjectured by Kurihara.
 \end{abstract}

\section{Introduction}

Let $D$ be the reduced discriminant of an indefinite quaternion algebra $B$ over $\Q$ and let $N \geq 1$ be a positive integer, prime to $D$. Let $X_0^D(N)/\Q$ denote the Shimura curve over $\Q$ attached to an Eichler order
of level $N$ in $B$.

As it is well-known, in the classical modular case automorphic forms of
$X_0(N):=X_0^1(N)$ admit Fourier expansions around the cusp of infinity. This allows to compute explicit
generators of the field of functions of such curves. Also, explicit methods are known to
determine bases of the space of their regular differentials, which are used to compute
equations for them and their quotients by Atkin-Lehner involutions.

In the general case, $D>1$, the question of writing down explicit equations of curves $X_0^D(N)$ over $\Q $
remains quite unapproachable. The absence of cusps has been an obstacle for explicit approaches
to Shimura curves. Ihara \cite{Ihara} was probably one of the first to express an interest on this problem, and already found an equation for the genus $0$ curve $X_0^6(1)$, while challenged to find others. Since then, several authors have contributed to this question (Kurihara \cite{Kur2}, Jordan \cite{Jor2}, Elkies \cite{Elk}, Clark-Voight \cite{Voi} for genus $0$ or/and $1$, Gonzalez-Rotger \cite{GoRo1}, \cite{GoRo2} for genus $1$ and $2$).

Elkies computes equations for the list of Shimura curves that he deals with using their hyperbolic (rather than the non-Archimedean uniformisations at primes dividing the discriminant) uniformisations. His method has the advantage that allows the identification of Heegner points in the equation, but is limited to very small discriminants $D$ and levels $N$.

The methods of Gonzalez-Rotger are heavily based on Cerednik-Drinfeld's theory for the special fiber at $p\mid D$ and the arithmetic properties of Heegner points. It allows to work with larger $D$ and $N$ but is again subjected to sever restrictions: the genus must be at most $2$ and $J_0^D(N)$ must be isogenous to a product of elliptic curves. In addition, this method does not allow to locate Heegner points in the given model of the curve.
The present paper is in the line of \cite{GoRo2} and one of the aims is removing such strong restrictions.


More precisely, the aim of this note is to introduce an algorithm to compute equations for hyperelliptic Shimura curves with good reduction at $2$. For the sake of simplicity we restrict ourselves to the case $N=1$ and write $X_0^D=X_0^D(1)$, although we believe that the procedure can be easily generalized to the case of arbitrary square-free $N$. Polynomials defining equations of hyperelliptic curves are closely related to their set of \emph{Weierstrass points}. The set of Weierstrass points $\WP(X_0^D)$ of a hyperelliptic Shimura curve $X_0^D$ turns out to be a disjoint union of Heegner points: $$\WP(X_0^D)=\bigsqcup_i \CM(R_i),$$ for suitable orders $R_i$ in imaginary quadratic fields. As a consequence, $X_0^D$ admits an equation of the form
\begin{equation}\label{eqintro}
y^2=\prod_i p_i(x),
\end{equation}
where $p_i(x)$ is a polynomial attached to each set of Heegner points $\CM(R_i)$.

Let $\cX_0^D$ denote Morita's integral model of $X_0^D$. Over $\Z[1/2]$, $\cX_0^D$ will also be defined by an equation of the form \eqref{eqintro}. As we shall explain in detail, the specialization of Weierstrass points at the special fiber of $\cX_0^D$ at a prime $p$ can be exploited in order to compute the $p$-adic valuation of the discriminants $\disc(p_i)$ and resultants $\Res(p_1,p_j)$ of the above polynomials. We will make use of the theory of specialization of Heegner points introduced in \cite{Mol} in order to obtain such information.

Moreover, by means of the classical theory of complex multiplication we can also compute the splitting fields of each $p_i$. Exploiting the theory developed by Gross-Zagier in \cite{GrossZagier} we can further compute the leading coefficients of each $p_i$, once we have fixed a pair of Heegner points at infinity.

As a combination of all this data, we are able to compute an explicit model \eqref{eqintro} for $X_0^D$. The only algorithmic limitation of this method relies on the fact that it exploits certain instructions which are currently implemented (e.g. in \emph{MAGMA}) only for small degree field extensions. As long as the genus increases, the degrees of the fields involved in the computation become so large that make it impossible to proceed with the algorithm.

In \S \ref{SemiSthyp} we recall basic facts about semi-stable hyperelliptic curves and the specialization of their Weierstrass points. In \S \ref{hypShi} we introduce Shimura curves with special emphasis to the finite list of them which are hyperelliptic. In \S \ref{singespCM} we describe the singular specialization of Heegner points and in \S \ref{compudesc} we give an explicit recipe to compute it in terms of Ribet bimodules. In \S \ref{Esp_h_R1} we exploit the moduli interpretation of Shimura curves in order to compute the supersingular specialization of a suitable set of Heegner points. This is a crucial step in the computation of the leading coefficients of the polynomials involved, once we have fixed a pair of such Heegner points at infinity. In \S \ref{Algorithm} we present our algorithm and we devote \S \ref{S-S exmpl} and \S \ref{Ex55} to exhibit two examples of its implementation.

Finally, in \S \ref{ALquo} we explain how to adapt the algorithm to quotients of Shimura curves by Atkin-Lehner involutions. The degrees of the fields involved in the computation in this case are smaller and, consequently, we are able to compute more examples.
In \S \ref{results} we present a list of equations of Shimura curves and Atkin-Lehner quotients obtained by means of the algorithms introduced in the previous sections. These equations were unknown until now and were conjectured by Kurihara in \cite{Kur}.

\section{Semi-stable hyperelliptic curves}\label{SemiSthyp}

Let $X$ be a smooth, geometrically connected, projective curve of genus $g>1$ defined over a field $k$. It is said  that $X$ is a \emph{hyperelliptic curve} over $k$ if there exists a finite separable morphism
$X\rightarrow \PP^1_{k}$ of degree 2.  Whenever there is no risk of confusion about the field $k$ we shall only say that $X$ is hyperelliptic.
This is equivalent to the existence of an involution $\omega$ defined over $k$ such that the quotient curve $X/\omega$ has genus 0 and $k$-rational points. When this is the case, this involution is unique and is called \emph{the hyperelliptic involution}.  Moreover, it is well known that there exist functions $x,y\in k(X)$ satisfying a relation of the type
\begin{equation}\label{eq}
y^2+Q(x)y+P(x)=0, \;\; P,Q\in k[x], \;\; 2g+1\leq \max\{\deg P,2\deg
Q\}\leq 2g+2,
\end{equation}
and such that the function field of $X$ is $k(X)=k(x,y)$.
 The hyperelliptic involution $\omega$ is then given by $(x,y)\ra(x,Q(x)-y)$ and, for the particular case that $\car (k)\neq 2$, we can take $Q(x)=0$.
The set of $\overline k$-rational points of $X$ consists of the set
of  affine points defined by \eqref{eq}
together with a $k$-rational point at infinity if  $\deg(Q(x)^2-4P(x))=2g+1$, or a
pair of points at infinity if  $\deg(Q(x)^2-4P(x))=2g+2$. In the later
case, both points are either $k$-rational or Galois conjugate over a quadratic extension of $k$.

We shall denote by $\WP(X)$  the set of Weierstrass points of $X$.
It coincides with the set of fixed points of $\omega$. Hence, $\WP(X)$ contains the point at infinity in case $\deg(Q^2(x)-4
P(x))=2g+1$, and  all points of the form $(\gamma, Q(\gamma)/2)$ or $ (\gamma,\sqrt{P(\gamma)})$, depending whether $\car (k)\neq 2$ or not, where
$\gamma$ is a root of $R(x)=Q^2(x)-4P(x)$.

If $k=\Q$, a \emph{Weierstrass model} for $X$ is a model $\cW$ over $\Z$, i.e. a normal fibered surface over $\Spec(\Z)$ with generic fiber $X$, such that $\omega$ can be extended to an involution on $\cW$, which we still denote by $\omega$, and the quotient $\cW/\langle\omega\rangle$ is smooth over $\Z$. We shall also denote by $\WP(\cW)$ the set of fixed points of $\omega$ on $\cW$.
By \cite[Remark 3.5]{Knu-Klei}, every smooth model of $\PP^1_\Q$ is isomorphic to $\PP^1_\Z$. Hence, any Weierstrass model $\cW$ satisfies $\cW/\langle\omega\rangle=\PP^1_\Z$ and, by \cite[Lemme 1]{liu96}, $\cW$ is the projective closure of the affine curve defined by:
\begin{equation}\label{eqZ}
y^2+Q(x)y+P(x)=0, \;\; P,Q\in \Z[x], \;\; 2g+1\leq \max\{\deg P,2\deg
Q\}\leq 2g+2.
\end{equation}
Given such a hyperelliptic equation, we define the \emph{discriminant of the Weierstrass model} as follows:
\begin{equation}\label{Disc}
\Delta(\cW)=\left\{\begin{array}{cc}
2^{-4(g+1)}\disc(R(x)) & \mbox{if}\;\deg R(x)=2g+2\,,\\
2^{-4(g+1)}c^2\disc(R(x)) & \mbox{if}\;\deg R(x)=2g+1\,,
\end{array}\right.
\end{equation}
where $R(x)=Q(x)^2-4P(x)$ and $c$ is its leading coefficient. The special fiber $\cW_p$ of $\cW$ at $p$ is smooth over $\F_p$ if and only if $p\nmid\Delta(\cW)$ (c.f. \cite{liu96}).

Assume now that $k$ is algebraically closed, let $C$ be an algebraic curve over $k$, and let $x\in C(k)$. We say that $x$ is an ordinary double point if
\begin{equation}\label{eqodp}
\widehat{O_{C,x}}\simeq k[[u,v]]/(uv)\simeq k[[u,v]]/(u^2-v^2),
\end{equation}
where $\widehat{ O_{C,x}}$ is the completion of the local ring $O_{C,x}$.
A curve $C$ over $k$ is said to be  \emph{semi-stable} if it is reduced and all its singular points are ordinary double points.

Let $S$ be an affine Dedekind scheme of dimension 1, with fraction field $K$.
Let $C$ be a normal, connected, projective curve over $K$. A \emph{model of $C$ over $S$} is a normal fibered surface $\mathcal{C}\ra S$ together with an isomorphism of its generic fiber $f:\mathcal{C}_\eta\ra C$.
We say that the model  $\mathcal C\ra S$ is \emph{semi-stable} if for each $s\in S$ the geometric fiber $\mathcal C_s\times_{k(s)}\overline{k(s)}$ is semi-stable over $\overline{k(s)}$, where $k(s)$ stands for the residue field of $S$ at $s$.

\begin{proposition}\cite[Corollary 10.3.22]{liu02}
Let $\mathcal C\rightarrow S$ be a semi-stable model of a curve $C$. Let $s\in S$, and let $x\in \mathcal C_s$ be a singular point of $\mathcal C_s$.
Then there exists a Dedekind scheme $S'$, étale over $S$, such that any point $x'\in \mathcal C':=\mathcal C\times_S S'$ above $x$ lying on $\mathcal C'_{s'}$ is an ordinary double point in $\mathcal C'_{s'}\ra \Spec (k(s'))$.
Moreover,
\[
\widehat{O_{\mathcal C',x'}}\cong \widehat{O_{S',s'}}[[u,v]]/(uv-c) \quad c\in \mathfrak{m}_{s'}O_{S',s},
\]
where $\widehat{O_{\mathcal C',x'}}$ and $\widehat{O_{S',s'}}$ are the completions of $O_{\mathcal C',x'}$ and $O_{S',s'}$ respectively.

If $C$ is smooth, then $c\neq 0$. Let $e_x$ be the normalized valuation of $c$ in $O_{S',s'}$, then $e_x$ does not depend on the scheme $S'$ chosen.
\end{proposition}

\begin{definition}
The value $e_x$ described in the above proposition  is called \emph{the thickness of the singularity} $x\in \mathcal C_s$.
\end{definition}


\begin{theorem}\label{EspSing}
Let $\cW\ra \Spec(\Z)$ be a Weierstrass semi-stable model, and let $p$
be an odd prime of bad reduction. Let $\tilde P\in \cW_p(\overline\F_p)$ be a
singular point lying in an affine open defined by an equation
$y^2+Q(x)y+P(x)=0$. Then, there exist exactly two Weierstrass points
$P_1,P_2\in\WP(X)$ that specialize to $\tilde P$. Moreover, the thickness
of $\tilde P$ is $e_{\tilde P}=2\nu(\gamma_1-\gamma_2)$, where $\nu$ is the
normalized valuation at $p$ and $\gamma_i$ are the roots of
$R(x)=Q(x)^2-4P(x)$ corresponding to $P_1$ and $P_2$.
\end{theorem}
To prove this result we need the following technical lemma.
\begin{lemma}\label{tec}
Let $A$ be a ring such that $n\in A^*$. Then $s=(1+t)^n-1\in A[t]$ satisfies $A[[t]]=A[[s]]$ and, moreover, there exists $ f(s)\in sA[[s]]$ such that $1+s=(1+f(s))^n$.
\end{lemma}
\begin{proof}
This is exercise 1.3.9 of \cite{liu02}. The proof is left to the reader.
\end{proof}
\begin{proof}[Proof of Theorem \ref{EspSing}]
First we shall prove that there are exactly two Weierstrass
points $P_1,P_2\in\WP(X)$ specializing to $\tilde P$. Write $\overline
\cW_p=\cW\times\Spec(\overline\F_p)$ for the geometric fiber of $\cW$ at
$p$. Since $p\neq 2$, an affine open $\mathcal{U}$ of $\overline
\cW_p$ shall be of the form
$\mathcal{U}=\Spec(\overline\F_p[x,y]/(y^2-\tilde R(x)))$, where
$\tilde R(x)$ is the reduction of $R(x)$ modulo $p$. Hence it is
clear that singularities of $\mathcal{U}$ correspond to multiple
roots of $\tilde R(x)$. Without loss of generality, assume $x=0$ is the multiple root of $\tilde R(x)$ corresponding to
$\tilde P$. We get $\tilde R(x)=x^m \tilde h(x)$, where $\tilde h(x)=\tilde h(0)(1+x \tilde r (x))$ and  $\tilde h(0)\neq 0$.
The local ring $O_{\overline \cW_p,\tilde P}$ at $\tilde P$ is given by:
\[
O_{\overline \cW_p,\tilde P}=(\overline \F_p[x,y]/(y^2-x^m\tilde
h(x)))_{(x,y)},
\]
and it follows that
\[
\widehat{O_{\overline \cW_p,\tilde P}}=\overline
\F_p[[x,y]]/(\frac{y^2}{\tilde h(x)}-x^m).
\]
By Lemma \ref{tec}, taking $A=\overline\F_p[[y]]$, $t=x \,\tilde r (x)$ and
$n=2$, we obtain that $\tilde h(x)$ is a square  in
$(\overline \F_p[[x,y]])^*$. Hence $\widehat{O_{\overline
\cW_p,\tilde P}}=A_m$, where
\[
A_m:=\overline \F_p[[x,y]]/(y^2-x^m),\;\;m\geq 2.
\]

Since $\cW$ is semi-stable,  $\overline \cW_p/\overline
\F_p$ must be semi-stable. Therefore $\tilde P$ is an ordinary double point and
$\widehat{O_{\overline \cW_p,\tilde P}}\simeq \overline
\F_p[[x,y]]/(y^2-x^2)=A_2$. From
$A_2\simeq A_m$,  it follows that  $m=2$. As a consequence, $\tilde P$ is attached to a
root $\tilde \gamma$ of $\tilde R(x)$ with multiplicity 2 and we
conclude that there exist exactly two $P_1,P_2\in\WP(X)$ that specialize to $P$
(attached to the roots $\gamma_1$ and $\gamma_2$ of $R(x)$ that
reduce to $\tilde \gamma$).

Next, we proceed to compute the thickness $e_{\tilde P}$ of $\tilde P$:
the equation $Y^2=R(x)=Q(x)^2-4P(x)$ defines $\cW$ in a neighborhood of $(p)\in\Spec(\Z)$.
After extending to a finite extension $k'\supseteq \F_p$ if necessary, we can suppose that any singular point $\tilde P'\in \cW_p\times\Spec(k')$ lying over $\tilde P$ is $k'$-rational. Without loss of generality, assume that $\tilde P'$ is defined by $x=0, Y=0$. That is,
\[\tilde R(x)=x^2\tilde h(x),\quad \tilde h(0)\neq 0.\]
We can choose an étale scheme $S'$ over $\Spec(\Z)$ and a point $\pi\in S'$ above $(p)$ such that $k'=\F_p(\pi')$. Notice that, if we write $\cW'=\cW\times_S S'$, the point $\tilde P'$ lies in $(\cW')_{\pi'}$ and its local ring is $\cO_{\cW',(\pi',\tilde P')}=(\cO_{S',\pi'}[x,Y]/(Y^2-R(x)))_{(x,Y)}$.

Let $\widehat{\cO_{S',\pi'}}$ be the completion of $\cO_{S',\pi'}$ and denote by $\nu$ its normalized valuation. Let us consider $R(x)$ over $\widehat{\cO_{S',\pi'}}$. Since its reduction is $\tilde R(x)=x^2\tilde h(x)$ with $\tilde h(0)\neq 0$, we apply the Classical Hensel's Lemma (cf.\cite{Riben}) to $x^2$ and $\tilde h(x)$ and we obtain that $R(x)=(x^2+ax+b)\cdot h(x)$, where $\nu(h(0))=0$, $\nu(a)>0$ and $\nu(b)>0$.
Extending $S'$ to a bigger étale $\Spec(\Z)$-scheme if necessary, we can suppose that $h(0)$ has a square root in $\widehat{\cO_{S',\pi'}}$. Since $h(x)=h(0)(1+x\cdot r(x))\in \widehat{\cO_{S',\pi'}}[[x]]^*$, by Lemma \ref{tec}, there exists $s(x)\in\widehat{\cO_{S',\pi'}}[[x]]^*$ such that $s(x)^2=h(x)$. Therefore
\[
\widehat{\cO_{\cW',(\pi',\tilde P')}}=\widehat{\cO_{S',\pi'}}[[x,Y]]/(Y^2-(x^2+ax+b)\cdot h(x))=\widehat{\cO_{S',\pi'}}[[x,Y]]/(\left(\frac{Y}{s(x)}\right)^2-\left(x+\frac{a}{2}\right)^2-\Delta),
\]
where $\Delta=a^2-4\cdot b$. Writing $u=Y/s(x)+x+a/2$ and $v=Y/s(x)-x-a/2$, we obtain that $\widehat{\cO_{S',\pi'}}[[x,Y]]=\widehat{\cO_{S',\pi'}}[[u,v]]$ and
\[
\widehat{\cO_{\cW',(\pi',\tilde P')}}=\widehat{\cO_{S',\pi'}}[[u,v]]/(u\cdot v-\Delta).
\]
Hence, we deduce that $e_{\tilde P}=\nu(\Delta)$.

Since the roots of the polynomial $x^2+ax+b$ are precisely the two
unique roots $\gamma_1,\gamma_2\in\overline \Q$ that reduce to
$\tilde \gamma$, and $\Delta$ is the discriminant of the polynomial
$x^2+ax+b$, it follows that $\Delta=(\gamma_1-\gamma_2)^2$ and
$e_{\tilde P}=2\nu(\gamma_1-\gamma_2)$.
\end{proof}

\section{Hyperelliptic Shimura curves}\label{hypShi}

Let $B$ be an indefinite division quaternion algebra over $\Q$ and let $\cO$ be a maximal order in $B$. By an abelian surface with quaternionic multiplication (QM) by $\cO$
over a field $K$ we mean a pair $(A, i)$ where:
\begin{itemize}
\item [i)] $A/K$ is an abelian surface.
\item [ii)]  $i:\cO\hookrightarrow\End(A)$ is an embedding.
\end{itemize}

For such a pair we denote by $\End (A,i)$ the ring of endomorphisms which commute with
$i$, i.e., $\End(A,i)=\{\phi\in\End(A):\phi i(\alpha)=i(\alpha)\phi\mbox{  for all }\alpha\in\cO\}$.
Two abelian surfaces $(A,i)$ and  $(A',i')$ with QM by $\cO$ are isomorphic if there is an isomorphism $\phi:A\rightarrow A'$
such that $\phi\circ i(\alpha)=i'(\alpha)\circ\phi$ for all $\alpha\in\cO$.
Throughout, we shall denote by $[A,i]$ the isomorphism class of $(A,i)$.

Let us denote by $X_0^D/\Q $ Shimura's canonical model of the Shimura curve associated to $\cO$. As Riemann surfaces,  $X_0^D(\C) = \Gamma_0^D\backslash \mathcal H$, where $\mathcal H$ is the Poincaré upper half plane and $\Gamma_0^D$ is the image of $\cO$ through the embedding $B\hookrightarrow B\otimes\R\simeq M(2,\R)$. As is well known, $X_0^D$ represents, as a coarse moduli space, the moduli problem of classifying abelian surfaces with quaternionic multiplication by $\cO$.
Hence an isomorphism class $P=[A,i]$ shall be often regarded as a point on $X_0^D$.

It follows from the work of Morita, Cerednik and Drinfeld that $X_0^D$ admits a proper integral model $\cX$ over $\Z$, smooth over $\Z[\frac{1}{D}]$, which suitably extends the moduli interpretation to arbitrary base schemes (cf.\cite{Mo},\cite{BC}).
Moreover, $\cX$ is semi-stable at every prime $p$ dividing $D$, and singular points of $\cX_p$ are in correspondence with certain algebraic objects (see correspondence \eqref{corrsing}),
from which we will recover their thicknesses (see Lemma \ref{espord}).

Let $K$ be an imaginary quadratic field and let $R$ be an order in $K$.
A point $P=[A,i]\in X_0^D(\C)$ is a Heegner (or CM) point by $R$ if $\End(A,i)\simeq R$. Throughout, we shall fix the isomorphism $R\simeq\End(A,i)$ to be the canonical one described in \cite[Definition 1.3.1]{Jor}. We denote by $ \CM(R)$ the set of Heegner points by $R$. By main Theorem I of \cite{Shim}, the extension $K(P)$ of $K$ generated by the coordinates any $P\in\CM(R)\subset X_0^D$ is the ring class field of $R$, $H_R$. Moreover, $[K(P):\Q (P)]$ is $1$ or $2$ and the number field $\Q
(P)$ can be determined, up to Galois conjugation (see Theorem 5.12 of \cite{GoRo1}).

For every divisor $m|D$ let us denote by $\omega_m$ the corresponding Atkin-Lehner involution on $X_0^D$, which is  defined over $\Q$. The property $\omega_m\cdot \omega_n=\omega_{m\cdot n/(m,n)^2}$ implies that the set $W(D)=\{\omega_d:d|D\}$ is a subgroup of automorphisms of $X_0^D$ isomorphic to $(\Z/2\Z)^{\#\{p\mid D\}}$. The action of these involutions on Heegner points can be found in Lemmas 5.9 and 5.10 of \cite{GoRo1} and, as the following result shows, their set of fixed points is also a set of Heegner points.


\begin{proposition}\cite[\S 1]{Ogg}\label{WP}
Let $m \mid D,\; m > 0$. The set $\mathfrak{F}_{\omega_m}$ of fixed points of
the Atkin-Lehner involution $\omega_m$ acting on $X_0^D$ is
\[
\mathfrak{F}_{\omega_m}=
\left\{\begin{array}{ll}
\CM(\Z[\sqrt{-1}]) \sqcup \CM(\Z[\sqrt{-2}]) & \rm{if}\; m = 2\\
\CM(\Z[\sqrt{-m}]) \sqcup \CM(\Z[\frac{1+\sqrt{-m}}{2}]) & \rm{if}\; m \equiv 3\mod 4\\
\CM(\Z[\sqrt{-m}]) & \rm{otherwise.}
\end{array}\right.
\]
\end{proposition}


Ogg determined in \cite{Ogg} the $24$ values of $D$ for which $X_0^D$ is hyperelliptic over $\overline{\Q}$ and proved that only for $21$ values  of them the corresponding curves $X_0^D$ are hyperelliptic over $\Q$. The aim of this paper is to give a procedure to compute equations for all these cases. Since those of genus $2$ were computed by J. Gonz\'alez and V. Rotger in \cite{GoRo2}, we assume that $X_0^D/\Q$ is hyperelliptic over $\Q$ of genus $g>2$.  We present the values of $D$ and the corresponding genera for the  remaining  $18$ cases:
$$
\begin{array}{c|l}
g & D\\ \hline
3 & 2\cdot 31,2\cdot 47,3 \cdot 13,3\cdot 17,   3\cdot 23 , 5\cdot 7 ,  5\cdot 11 \\[2pt]
4&  2\cdot 37,
 2\cdot 43 \\[2pt]
5& 3\cdot 29 \\[2pt]
6& 2\cdot 67 \\[2pt]
7&  2\cdot 73 ,3\cdot 37,5\cdot 19 \\[2pt]
9& 2\cdot 97, 2\cdot 103, 3\cdot 53, 7\cdot 17
 \end{array}
$$
\begin{centerline}{\bf Table 1}\end{centerline}

\vskip 0.1 cm

The hyperelliptic involution $\omega$ of $X_0^D$ in all these cases
turns out to be the Atkin-Lehner involution $\omega_D$. Since the action of
$\omega_D$ has an interpretation in terms of the moduli problem, it
can be extended to an involution on the integral model $\cX$.
Moreover, we have an explicit description of the fibers $\cX_p$ and
the action of $\omega=\omega_D$ on them. Hence we can easily check
whether the quotient $\cX/\langle\omega\rangle$ is smooth over $\Z$.
If $\cX/\langle\omega\rangle$ is not smooth over $\Z$, then $\cX$ is not a Weierstrass model for $X_0^D$. Sometimes it is possible to
blow-down certain exceptional irreducible components in order to obtain a model
$\cW$ such that $\cW/\langle\omega\rangle$ is smooth over $\Z$ and,
thus,  defined
by an equation of the form \eqref{eqZ}:
\[
\cW:y^2+Q(x)y+P(x)=0, \quad P,Q\in\Z[x],\quad 2g+1\leq\max\{2\deg(Q),\deg(P)\}\leq 2g+2.
\]

\begin{remark}\label{NoWeierstrass}
But this is not always possible.
For example, the special fiber of Morita's integral model of $X_0^{87}$ at $p=29$ has the following form:

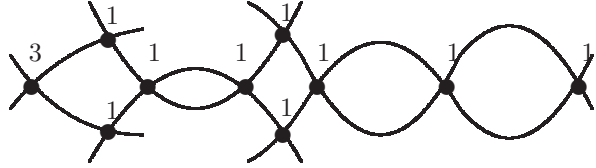
\begin{picture}(250,70)
\thicklines

\qbezier(260,20)(285,-10)(310,40)

\qbezier(260,40)(285,70)(310,20)

\qbezier(200,40)(230,-20)(260,40)

\qbezier(200,20)(230,70)(260,20)

\qbezier(180,60)(190,55)(200,40)

\qbezier(180,0)(190,5)(200,20)

\qbezier(120,60)(160,-20)(200,60)

\qbezier(120,0)(160,70)(200,0)

\qbezier(90,40)(110,10)(140,10)

\qbezier(90,20)(110,45)(140,50)

\put(98,28){\circle*{6}}

\put(127,46){\circle*{6}}

\put(127,11){\circle*{6}}

\put(142,28){\circle*{6}}

\put(193,48){\circle*{6}}

\put(193,10){\circle*{6}}

\put(179,28){\circle*{6}}

\put(206,28){\circle*{6}}

\put(255,28){\circle*{6}}

\put(305,28){\circle*{6}}

\put(126,52){1}

\put(126,16){1}

\put(192,53){1}

\put(192,17){1}

\put(142,38){1}

\put(175,38){1}

\put(206,38){1}

\put(255,38){1}

\put(306,38){1}

\put(97,38){3}

\end{picture}

\quad

Clearly, by blowing-down exceptional divisors it is not possible to obtain a fiber $\cW_p$ such that $\cW_p/\langle\omega\rangle$ is smooth over $\F_p$.
\end{remark}

In order to obtain explicit equations, we will focus our attention in two directions:

\vskip 0.1 cm {\it 1. Determination of the thicknesses of
Weierstrass points at every prime $p|D$.} Since the hyperelliptic
involution is the Atkin-Lehner involution $\omega_D$, we have that
$\WP(\cW)=\bigsqcup_i \CM(R_i)$, where $\{R_i\}$ is the set of the
orders in the imaginary quadratic field $K=\Q(\sqrt{-D})$ containing
the order $\Z(\sqrt{-D})$. By Theorem \ref{EspSing}, thicknesses
of singular specializations of $\WP(\cW)$ are related with roots
of the polynomial $R(x)=P(x)^2-4Q(x)$. In \S \ref{singespCM} we
shall discuss singular specialization of Heegner points and we shall
give an explicit recipe to obtain such thicknesses.

\vskip 0.1 cm
{\it 2. Determination of the leading coefficient of $R(x)=P(x)^2-4Q(x)$.}

Given the Weierstrass model $\cW$ of $X_0^D$, let $\mathcal U$ be the
affine open defined by the equation $y^2+P(x)y+Q(x)=0$. The set of
\emph{points at infinity} of $\mathcal U$ is the set of geometric
points of the generic fiber of $\cW\setminus\mathcal U$.
 Since Shimura curves do not have real points (cf. \cite[Proposition 4.4]{Shimura}), this set corresponds to a pair of conjugate points living in a quadratic extension of $\Q$ such that the hyperelliptic involution acts on them via the unique non-trivial Galois conjugation. In particular, this implies that $\deg(P^2-4Q)=2g+2$.
In order to fix a hyperelliptic equation of $\cW$, we must
choose a pair of points defined over an imaginary quadratic field
such that the hyperelliptic involution acts suitably on them.

It turns out that for every value $D$ in Table 1, there exists a maximal
order $R_\infty$ in an imaginary quadratic field $K_\infty$ with number class
$h_{R_\infty}=1$, i.e. $K_\infty=H_{R_\infty}$, discriminant coprime to $D$ and such that $\CM(R_\infty)\neq \emptyset$.  By
\cite[Lemma 5.10]{GoRo1}, complex conjugation
acts on every $P_\infty\in\CM(R_\infty)$ as the hyperelliptic involution
$\omega_D$. We fix $P_\infty\in\CM(R_{\infty})$ and we choose the
set $\{P_\infty,\omega_D(P_\infty)\}$ to be our set
of points at infinity. This choice  shall fix a hyperelliptic
equation $y^2+P(x)y+Q(x)=0$ of $\cW$, up to transformations of the
form $(x,y)\mapsto (x+a,y+h(x)),\;a\in\Z,\;h(x)\in\Z[x]$,
$\deg(h(x))\leq g+1$.

Our goal  is to determine the leading coefficient $a_R$ of the
polynomial $R(x)=P(x)^2-4Q(x)$. As a first approach, recall that the
field of definition of $P_\infty$ is $K_\infty=\Q(\sqrt{a_R})$.
Moreover, a prime $p$ divides $a_R$ if and only if $P_\infty$ and
$\omega_D (P_\infty)$ specialize to the same $\F_p$-rational Weierstrass point. Hence, the determination of the
specialization of these specific Heegner points will give a valuable
information about  the leading coefficient $a_R\in\Z$.

Since any $p\mid D$ is inert in $R_\infty$, $P_\infty$ has good reduction at $p$. Any Weierstrass point has singular specialization at any prime dividing $D$, hence $(a_R,D)=1$.
In order to determine the  remaining $p$-adic valuations of $a_{R}$, we introduce the following definition:
\begin{definition}
Let $R$ be a local valuation ring with uniformizer $\pi$.
\emph{The intersection index} of two ideals $I_1$ and $I_2$ of an
algebra $A$ over $R$ is the length of the algebra $A/(I_1 + I_2)$.
\end{definition}

Let $P_1$ and $P_2$ be the points in $\Spec(A)$ defined by $I_1$ and
$I_2$. By \cite[Lemma 3.13]{Sad}, the intersection index of $I_1$ and $I_2$
measures the maximal power $n$ of $\pi$ in which their inverse image
$\tilde P_1$ and $\tilde P_2$ coincide in $\Spec(A\otimes_R
(R/\pi^nR))$.

Recall that $P_\infty$ lies in the affine open defined by the relation $z^2+Q_1(v)z+P_1(v)=0$, where $Q_1(v)=v^{g+1}Q(1/v)$ and $P_1(v)=v^{2g+2}P(1/v)$. Moreover, the ideals defining $P_\infty$ and $\omega_D(P_\infty)$ are
\[
I_{P_\infty}=\langle v, z+\frac{Q_1(0)+\sqrt{a_R}}{2}\rangle,\quad
I_{\omega_D (P_\infty)}=\langle v,
z+\frac{Q_1(0)-\sqrt{a_R}}{2}\rangle.
\]
Set $K_p=K_\infty\otimes_\Q\Q_p$,  let $K_p^{\rm{unr}}$ be the
maximal unramified extension of $K_p$ and let $R_p^{\rm{unr}}$ be
its integer ring with uniformizer $\pi$. Write $\cW_p^{\rm{unr}}$
for the extension of scalars $\cW\times\Spec(R_p^{\rm{unr}})$ and
denote also by $P_\infty$ and $\omega_D (P_\infty)$ their inverse
image in $\cW_p^{\rm{unr}}$. Write $(P_\infty,\omega_D(P_\infty))_p$ for the intersection index between $P_\infty$ and $\omega_D(P_\infty)$ in $\cW_p^{\rm{unr}}$. Then, it is easy to check that $(P_\infty,\omega_D(P_\infty))_p$ is precisely $\nu_p(a_R)$, if $p$ ramifies or splits in $K_\infty$, and $\nu_p(a_R)/2$, if $p$ is inert in $K_\infty$.

Assume that $p\nmid D$. Since $\cX/\Z$ is the coarse moduli space
associated to the algebraic stack that classifies abelian surfaces
with QM by $\cO$ over any arbitrary base scheme (cf. \cite{BC}) and
$\cW_p^{\rm{unr}}=\cX_p^{\rm{unr}}$, this intersection index can be
interpreted in terms of the algebraic objects classified by
$P_\infty=[A_\infty,i_\infty]$ and $\omega_D
(P_\infty)=[A_\infty',i_\infty']$. Namely,
\begin{equation}\label{contact}
(P_\infty,\omega_D (P_\infty))_p:=\max\{n\geq 1:\;(A_\infty,i_\infty)\simeq (A_\infty',i_\infty')\;\; \rm{over}\;\;R_p^{\rm{unr}}/\pi^n R_p^{\rm{unr}} \}.
\end{equation}

In section \S \ref{Esp_h_R1} we describe the specialization of those Heegner points $P\in\CM(R)$ with class number $h_R=1$ and we provide a description of $(P,\omega_D (P))_p$ in purely algebraic and computable terms.

\section{Specialization of Heegner points}\label{singespCM}

For any two square-free positive integers $d$ and $n$ let $\Pic(d,n)$ stand for the set of isomorphism classes of oriented Eichler orders of level $n$ in a quaternion algebra of discriminant $d$ (see \cite[\S 2.1]{Mol} for the definition of oriented Eichler order).

Let $\cX/\Z$ be Morita's integral model of $X_0^D$ as above. Let $p\mid D$ be a prime of bad reduction of $\cX$.
Thanks to the work of Cerednik and Drinfeld (cf. \cite{Cer},\cite{Drin}), we know that the special fiber $\cX_p$ at $p$ is semi-stable. Moreover, its sets of singular points $(\cX_p)_{\rm{sing}}$ and irreducible components $(\cX_p)_{c}$ are in one-to-one correspondence with the sets $\Pic(\frac{D}{p},p)$ and two copies of $\Pic(\frac{D}{p},1)$, respectively. We shall denote by
\begin{equation}\label{corrsing}
\varepsilon_s:(\cX_p)_{\rm{sing}}\stackrel{1:1}{\longleftrightarrow}\Pic(D/p,p)
\end{equation}
and
\begin{equation}\label{corrcc}
\varepsilon_c:(\cX_p)_{c}\stackrel{1:1}{\longleftrightarrow}\Pic(D/p,1)\sqcup\Pic(D/p,1)
\end{equation}
the corresponding bijections.

For any $\tilde P=[\tilde A,\tilde i]\in(\cX_p)_{\rm{sing}}$, the endomorphism ring $\End(\tilde A,\tilde i)$ is an Eichler order of level $p$ in a definite quaternion algebra of discriminant $D/p$, equipped with natural orientations \cite[Proposition 2.1]{Rib}, hence its isomorphism class can be regarded as an element of $\Pic(\frac{D}{p},p)$. Moreover, one can see in \cite[\S 5]{Mol} that $\varepsilon_s(\tilde P)=\End(\tilde A,\tilde i)$.

\begin{lemma}\cite[\S 3]{Edix}\label{espord}
The thickness $e_{\tilde P}$ of any $\tilde P\in(\cX_p)_{\rm{sing}}$ is given by $e_{\tilde P}=\epsilon(\varepsilon_s(\tilde P))$, where $\epsilon:\Pic(D/p,p)\ra \Z$ stands for the natural map
\begin{equation}\label{epsilon}
\epsilon(\cO_i)=\#(\cO_i^*/\langle\pm 1\rangle ),\quad\mbox{for all }\cO_i\in\Pic(D/p,p).
\end{equation}
\end{lemma}

We proceed to introduce the concept of optimal embedding. It shall be useful for future computations since Heegner points are in correspondence with certain optimal embeddings. Throughout, for any $\Z $-algebra $\mathcal D$, write $\mathcal D^0=\mathcal D\otimes _{\Z} \Q $.
\begin{definition}
Let $\cO_{d,n}$ be an oriented Eichler order in $\Pic(d,n)$ and let $R$ be an order in an imaginary quadratic field $K$. An \emph{optimal embedding} with
respect to $R$ is a ring monomorphism  $\varphi :K\hookrightarrow \cO_{d,n}^0$ such that $\varphi(K)\cap\cO_{d,n}=\varphi(R)$.
For any oriented Eichler order $\cO_{d,n}$, let $\CM_{\cO_{d,n}}(R)$ denote the set of optimal embeddings $\varphi: R\hookrightarrow \cO_{d,n}$, up to conjugation by $\cO_{d,n}^*$. Let $\CM_{d,n}(R) = \sqcup_{\cO_{d,n} \in \Pic(d, n)} \CM_{\cO_{d,n}}(R)$, where
$\cO_{d,n}\in\Pic(d,n)$ runs over a set of representatives of oriented Eichler orders.
\end{definition}
It is well known (see \cite[\S 2.2]{Mol}) that there is a one-to-one correspondence between the set $\CM(R)$ and the set of optimal embeddings $\CM_{D,1}(R)$. We denote this correspondence by:
\begin{equation}\label{Phi}
\begin{array}{ccc}
\varphi:\CM(R) & \longrightarrow & \CM_{D,1}(R)\\
P & \longmapsto & \varphi(P).
\end{array}
\end{equation}

Let $\Pic(R)$ be the Picard group of $R$, i.e. the group of isomorphism
classes of projective $R$-modules of rank 1. Let
$\Phi_R:\Pic(R)\ra\Gal(H_R/K)$ be the group isomorphism given by
Artin's reciprocity map. Recall that all $P\in\CM(R)$ are defined over $H_R$.

As is well known (cf. \cite[\S 5]{Vig}), there is a faithful action
of $\Pic(R)$ on $\CM_{d,n}(R)$. For any $[J]\in\Pic(R)$ and
$\psi\in\CM_{d,n}(R)$, denote such action by $[J]\ast\psi$.
The following
theorem, known as the \emph{Shimura reciprocity law},  describes the Galois
action of $\Gal(H_R/K)$ in terms of the action of $\Pic(R)$ on
$\CM_{D,1}(R)$, via the correspondence of \eqref{Phi}:
\begin{theorem}\cite[Main Theorem I]{Shim}\label{theoryCM}
Let $P\in \CM(R)\subset X_0^D(H_R)$.
Then, if $[J]\in\Pic(R)$,
\[
[J]^{-1}\ast \varphi(P)=\varphi(P^{\Phi_R([J])}).
\]
\end{theorem}

Fix an algebraic closure $\F$ of $\F_p$. We proceed to describe the specialization map
\begin{equation}\label{espmap}
\Pi:X_0^D(\overline\Q)\ra \cX_p(\F),
\end{equation}
focusing on the specialization of Heegner points.
Let $P=[A,i]\in X_0^D( \bar \Q)$. Pick a field of definition $H$ of $(A,i)$. Fix a prime $\dP$ of $H$ above $p$ and let $\tilde A$ be the specialization of $A$ at $\dP$. By \cite[Theorem 3]{Rib2}, $A$ has potential good reduction. Hence after extending $H$ if necessary, we obtain that $\tilde A$ is smooth over $\F$. The pair $(\tilde A,\tilde i)$, where the embedding $\tilde i$ stands for the composition
$\cO\stackrel{i}{\hookrightarrow} \End(A)\hookrightarrow \End(\tilde A)$,
defines an abelian surface with quaternionic multiplication by $\cO$. Moreover, $\Pi(P)=[\tilde A,\tilde i]\in\cX_p(\F)$ is the specialization of $P$.

Let $P=[A,i]\in\CM(R)$ be a Heegner point. By \cite[Lemma 4.1]{Mol}, when $P$ specializes to a singular point the natural map $\phi_P:\End(A,i)\rightarrow\End(\tilde A,\tilde i)$ turns out to be an optimal embedding in $\CM_{D/p,p}(R)$. If instead $P$ has non-singular specialization, modifying the embedding $\End(A,i)\hookrightarrow\End(\tilde A,\tilde i)$ as in \cite[\S 5]{Mol} one obtains an optimal embedding $\phi_P^c\in\CM_{D/p,1}(R)$. In both cases, the isomorphism class of their target, that lies in $\Pic(D/p,p)$ and $\Pic(D/p,1)$ respectively, characterizes the singular point or the irreducible component where $P$ lies.

The following result describes the specialization of the point $P$ in terms of the behavior of $p$ in $K=R\otimes \Q$, and relates the action of $\Pic(R)\simeq\Gal(H_R/K)$ on $P$ with the corresponding ones on $\phi_P$ and $\phi_P^c$.

\begin{theorem}\label{espsing}
Let $P=[A,i]\in\CM(R)$. Then
$\Pi(P)=[\tilde A,\tilde i]\in(\cX_p)_{\rm{sing}}$ if and only if $p$ ramifies in $K$. In this case, the assignation $P\mapsto\phi_P$ defines a bijective map
\begin{equation}\label{reds}
\phi_s:\CM(R)\longrightarrow\CM_{\frac{D}{p},p}(R)
\end{equation}
satisfying $\phi_s(P^{\Phi_R([J])})=[J]\ast\phi_s(P)$ for all $[J]\in\Pic(R)$. Moreover, if $\Pi(P)\not\in(\cX_p)_{\rm{sing}}$, the assignation $P\mapsto\phi^c_P$ defines a bijective map
\begin{equation}\label{redcc}
\phi_c:\CM(R)\longrightarrow\CM_{\frac{D}{p},1}(R)\sqcup\CM_{\frac{D}{p},1}(R)
\end{equation}
satisfying $\phi_c(P^{\Phi_R([J])})=[J]\ast\phi_c(P)$ for all $[J]\in\Pic(R)$
\end{theorem}
\begin{proof}
Combine \cite[Theorem 5.3]{Mol}, \cite[Theorem 5.4]{Mol} and \cite[Theorem 5.8]{Mol} with Theorem \ref{theoryCM}.
\end{proof}

\begin{remark}\label{remespsing}
Let $\pi:\CM_{D/p,p}(R)\ra\Pic(D/p,p)$ and $\pi':\CM_{D/p,1}(R)\sqcup\CM_{D/p,1}(R)\ra\Pic(D/p,1)\sqcup\Pic(D/p,1)$ be the natural forgetful projections that map a conjugacy class of optimal embeddings  $\varphi:R\hookrightarrow \cO_{i}$ to the isomorphism class of its target $\cO_{i}$. Notice that, if the specialization $\Pi(P)$ lies in $(\cX_p)_{\rm{sing}}$, such specialization is characterized by
$\varepsilon_s(\Pi(P))=\pi(\phi_s(P))$. On the other hand, if $\Pi(P)\not\in(\cX_p)_{\rm{sing}}$, then
$\varepsilon_c(\Pi(P))=\pi'(\phi_c(P))$.
\end{remark}

If we are able to compute the map $\phi_s$ explicitly, by Lemma \ref{espord} we shall obtain the thickness of the singular specialization of any Heegner point $P\in\CM(R)$ through the rule:
\begin{equation}\label{espsingCM}
e_{\Pi(P)}=\epsilon(\pi(\phi_s(P))).
\end{equation}
Once we know  the specialization and the thickness of a singular Heegner point in $\cX$, we can easily determine its specialization and its thickness in
$\cW$. Indeed, if
$pr:\cX\ra\cW$ is the blown-down map, then  the thickness of a singular
point $\tilde P\in\cW$ is:
\begin{equation}\label{thickblow_d}
e_{\tilde P}=\sum_{\tilde Q\in\cX_p,\;pr(\tilde Q)=\tilde
P}e_{\tilde Q}+\#\{\mathcal{C}\mbox{ connected
component},\;pr(\mathcal{C})=\tilde P\}-1
\end{equation}

\section{Computable description of the CM map $\phi_s$}\label{compudesc}

In order to give a computable description of the map $\phi_s$ we
shall introduce the concept of $(\cO,\cS)$-bimodule. We will see
that the specialization of any point $P\in\CM(R)$ is characterized
by a certain bimodule and the optimal embedding $\phi_s(P)$ can be
described in purely algebraic terms.

Let $p$ be a prime and let $\cS\in\Pic(p,1)$.
An $(\cO,\cS)$-bimodule $\cM$ is a free module of rank 4 over $\Z$ endowed with structures of left $\cO$-module and right $\cS$-module. 
The $(\cO,\cS)$-bimodules were introduced by Ribet in \cite{Rib} and they provide a useful tool for the analysis of certain supersingular points on the fiber $\cX_p$, as we now describe.

Let $\tilde P=[\tilde A,\tilde i]\in\cX_p(\F)$ such that $\tilde A$
is isomorphic to the product of two supersingular elliptic curves.
By \cite[Theorem 3.5]{Shi}, $\tilde A\simeq \tilde E^2$ for any fixed
supersingular elliptic curve $\tilde E$ over $\F$. Let $\cS$ be the
endomorphism ring of $\tilde E$. Then $\cS$ is a maximal order in a
definite quaternion algebra of discriminant $p$. By \cite[p. 37]{Rib}, $\cS$ comes equipped with a
natural orientation at $p$ and therefore can
be regarded as an element of $\Pic(p,1)$. Hence, giving such an
abelian surface $(\tilde A,\tilde i)$ with QM by $\cO$ is equivalent
to providing an optimal embedding
\[
\tilde i:\cO\hookrightarrow M(2,\cS)\simeq\End(\tilde A).
\]
Moreover, such a map provides a left $\cO$-module structure on the
right $\cS$-module $\cM_{\tilde P}=\cS\times\cS$. Since
$\cS\times\cS$ is free of rank 4 over $\Z$, $\cM_{\tilde P}$ defines
an $(\cO,\cS)$-bimodule.

Given $\tilde P=[\tilde A,\tilde i]\in\cX_p(\F)$ as above, one can compute the endomorphism ring $\End(\tilde A,\tilde i)$ in terms of the bimodule $\cM_{\tilde P}$. Let $\End_\cO^\cS(\cM_{\tilde P})$ be the set of $(\cO,\cS)$-module endomorphism of $\cM_{\tilde P}$, i.e., $\Z$-endomorphisms which are equivariant for the left action of $\cO$ and the right action of $\cS$. Then it is easy to check that $\End(\tilde A,\tilde i)=\End_\cO^\cS(\cM_{\tilde P})$ (cf.\cite[p.7]{Mol}).

Let $P=[A,i]\in\CM(R)$ be a Heegner point and assume that $p\mid D$. It follows from \cite[\S 4]{Rib} that if $\Pi(P)=[\tilde A,\tilde i]\in(\cX_p)_{\rm{sing}}$, the abelian surface $\tilde A$ is isomorphic to a product of supersingular elliptic curves. It thus makes sense to consider its attached bimodule $\cM_{\Pi(P)}$.
Next theorem allows us to describe the maps $\phi_s$ in terms of the $(\cO,\cS)$-bimodule $\cM_{\Pi(P)}$.
\begin{theorem}\label{redbim}\cite[Theorem 4.2]{Mol} Let $P=[A,i]\in\CM(R)$ be a Heegner point and let $(\varphi(P):R\hookrightarrow\cO)\in\CM_{D,1}(R)$ be its corresponding optimal embedding. Assume that $p\mid D$ and $\Pi(P)\in(\cX_p)_{\rm{sing}}$. Then,
\begin{itemize}
\item [(a)]There exists an optimal embedding $\psi_p:R\hookrightarrow \cS$ such that
\begin{equation}\label{bimodul}
\cM_{\Pi(P)}=\cO\otimes_R\cS,
\end{equation}
where $\cS$ is regarded as left $R$-module via $\psi_p$ and $\cO$ as right $R$-module via $\varphi(P)$.

\item [(b)] The optimal embedding $\phi_s(P)$ is given by the rule
\begin{equation}\label{rulephi}
\begin{array}{rcc}
R & \hookrightarrow & \End_{\cO}^{\cS}(\cO\otimes_R\cS) \\
\delta & \longmapsto & \alpha\otimes s  \mapsto  \alpha\delta\otimes s,
 \end{array}
\end{equation}
up to conjugation by $\End_{\cO}^{\cS}(\cO\otimes_R\cS)^{\times}$.
\end{itemize}
\end{theorem}

\begin{remark}
The embedding $\psi_p:R\hookrightarrow \cS$ depends on the immersion $\rho:H_R\hookrightarrow \overline\Q_p$ chosen for the specialization.
Given another optimal embedding $\psi_p'\in\CM_{p,1}(R)$, there exists a different immersion $\rho':H_R\hookrightarrow \overline\Q_p$ such that, specializing via $\rho'$, Theorem \ref{redbim} applies with $\psi_p'$ instead of $\psi_p$.
Indeed, as one can see in \cite[\S 4]{Mol}, the embedding $\psi_p$ corresponds to the inclusion $\End(E)\hookrightarrow\End(\tilde E)$, where $E$ is the CM elliptic curve $\C/R$ and $\tilde E$ is its specialization via $\rho$. Both, the set of immersions $H_R\hookrightarrow \overline\Q_p$ and $\CM_{p,1}(R)$ are $\Pic(R)$-torsors and the action of $\sigma\in\Gal(H_R/K)\simeq\Pic(R)$ turns $\psi_p$ into the embedding $\End(E^\sigma)\hookrightarrow\End(\tilde E^\sigma)$, where $\tilde E^\sigma$ is specialized by means of $\rho$. It is clear that such an optimal embedding coincides with the one obtained specializing $E$ via $\rho'=\rho\circ\sigma$.
\end{remark}

Since the above theorem describes the map $\phi_s$ in terms of purely algebraic objects, we shall be able to compute the image $\phi_s(P)$ starting from the corresponding embedding $\varphi(P)\in\CM_{D,1}(R)$ of \eqref{Phi}. Next,  we shall present an explicit description  of  $\phi_s(P)$ obtained from an explicit description of  $\varphi(P)$.

\begin{definition}\label{quatercomp}
Given a quaternion algebra $B$, an imaginary quadratic field $K$ and
an embedding $\psi :K\hookrightarrow B$, \emph{the
quaternionic complement of $\psi(K)$} is the set
\[
\psi(K)_-=\{\alpha\in
B\colon\;\alpha\psi(x)=\psi(x^\sigma)\alpha,\;\;\mbox{for all
}x\in K\},
\]
where $\sigma$ is the single non-trivial element of $\Gal(K/\Q)$.
By \cite[\S 1]{Vig}, $\psi(K)_-$ is a $K$-vector space of dimension 1. We
sometimes refer the element of a basis as a \emph{quaternionic
complement of} $\psi$. It is an element $j\in B$ such that
$j\psi(x)=\psi(x^\sigma)j,\;\;\mbox{for all }x\in K$ and $j^2\in\Q$.
\end{definition}

Let $(\psi:R\hookrightarrow \cO_{d,n})\in\CM_{d,n}(R)$ be an optimal
embedding. The free right $R$-module structure of $\cO_{d,n}$ given
by $\psi$ provides a decomposition $\cO_{d,n}\simeq R\oplus e I$,
where $I$ is a locally free $R$-module and $e\in \cO_{d,n}^0$. This
decomposition determines completely $\psi$. On the other hand, $\cO_{d,n}^0$
is characterized by the presentation $\cO_{d,n}^0\simeq K\oplus j
K$, where $\cO_{d,n}^0$ is also regarded as a right $K$-vector space
via $\psi$ and $jK$ is the quaternionic complement of
$K\stackrel{\psi}{\hookrightarrow}\cO_{d,n}^0$. Recall that $j$ is
determined by $j^2\in\Q$ and the fact that
$j\psi(x)=\psi(x^\sigma)j,\;\;\mbox{for all }x\in K$.

In conclusion, in order to compute $(\phi_s(P):R\hookrightarrow
\Lambda)\in\CM_{D/p,Np}(R)$ explicitly, we only have to present the
corresponding decompositions of $\Lambda$ and $\Lambda^0$ via
$\phi_s(P)$.

\begin{theorem}\label{compphi} Let $P\in\CM(R)$ be a Heegner point and assume that $p\mid D$ and $\Pi(P)\in(\cX_p)_{\rm{sing}}$.
Let $(\psi_p :R\hookrightarrow\cS)\in \CM_{p,1}(R)$ be the fixed
optimal embedding of Theorem \ref{redbim}. Write $\cS^0=H$ and let $H=K\oplus
j_2K$, $\cS\simeq R\oplus e_2I_2$, $j_2^2=m_2$,
$e_2=e_{2,1}+j_2e_{2,2}$, be the presentations of $H$ and $\cS$
induced by $\psi_p$. Analogously, let $B=K\oplus j_1K$ and
$\cO\simeq R\oplus e_1I_1$, $j_1^2=m_1$, $e_1=e_{1,1}+j_1e_{1,2}$,
be the presentations of $B$ and $\cO$ induced by $\varphi(P)$. Then,
the optimal embedding $\phi_s(P):R\hookrightarrow \Lambda$ is
characterized by:
$$
\Lambda^0=K\oplus j_3 K \quad \text{and }\quad \Lambda=R\oplus
e_3I_3\,,
$$
where $j_3$  is a quaternionic complement of $\phi_s(P)$  such that
$j_3^2=m_1\cdot m_2$, $e_3=e_{2,1}\cdot
e_{1,1}^\sigma-j_3e_{2,2}\cdot e_{1,2}^\sigma$ and
$$
I_3=\left\{\begin{array}{ll}
I_2I_1^\sigma & \mbox{if } {e_{1,1}}=0,e_{2,1}=0,\\
I_2I_1^\sigma\cap \frac{1}{e_{1,1}^\sigma}I_2 & \mbox{if }{e_{1,1}}\neq 0,e_{2,1}=0,\\
(I_2\cap \frac{1}{e_{2,1}}R)I_1^\sigma & \mbox{if } {e_{1,1}}=0,e_{2,1}\neq 0,\\
(I_2\cap \frac{1}{e_{2,1}}R)I_1^\sigma\cap
\frac{1}{e_{1,1}^\sigma}I_2 & \mbox{if }{e_{1,1}}\neq 0,e_{2,1}\neq
0.
\end{array}\right.
$$
\end{theorem}

\begin{proof}
Attached to the right $K$-module structure of $B$ via $\varphi(P)$
we have two distinct basis, namely $\langle1,j_1\rangle$ and
$\langle1,e_1\rangle$. We denote by
$M_{e_1}=\left(\begin{array}{cc}1&e_{1,1}\\0&e_{1,2}\end{array}\right)$
the matrix attached to the change of basis.

It follows that an element $z=x+j_1y\in K\oplus j_1K=B$ acts
on $K\oplus j_1K$ via the matrix
\[
M_z=\left(\begin{array}{cc}x & m_1y^\sigma\\y &
x^\sigma\end{array}\right)\in M_2(K).
\]

Since $B\otimes_K H=(K\oplus j_1K)\otimes_K H= H\oplus j_1 H$, any
element $z=x+j_1y\in  B$ acts on $B\otimes_K H$ through the same
matrix $M_z$. Hence
\[
\Lambda^0=\End_B^H(B\otimes_K H)=\left\{\left(\begin{array}{cc}a &
b\\c & d\end{array}\right)\in M_2(H)\;/\;M_z\left(\begin{array}{cc}a
& b\\c & d\end{array}\right)=\left(\begin{array}{cc}a & b\\c &
d\end{array}\right)M_z\right\}.
\]
This implies that $a=d$, $m_1 c=b$, $a\in\{x\in H
\colon\;xy=yx,\;\mbox{ for all }y\in K\}=K$ and $b\in\{x\in H
\colon\;xy=y^\sigma x,\;\mbox{ for all } y\in K\}=j_2K$. Thus
\begin{equation}\label{eqi3}
\Lambda^0= K \oplus \left(\begin{array}{cc}0 & m_1j_2\\j_2 &
0\end{array}\right)K = K\oplus j_3K,\quad
j_3=\left(\begin{array}{cc}0 & m_1j_2\\j_2 & 0\end{array}\right)
\end{equation}
where $j_3$ satisfies $xj_3=j_3x^\sigma$ for all $x\in K$ and
$j_3^2=m_1m_2\in\Q$. Hence $j_3$ is a quaternionic complement of
$\phi_s(P):K\hookrightarrow\Lambda^0$.

The $R$-module decomposition $\cO=R\oplus e_1I_1$ yields the
$\cS$-module structure of $\cO\otimes_R\cS$ as $\cS\times
(I_1\otimes_R\cS)$ with basis $\langle 1,e_1\rangle$. We turn it
into our original basis $\langle 1,j_1\rangle$ by means of $M_{e_1}$. Then,
\[
\Lambda=\{(a+j_3b)\in
\Lambda\otimes_{\Z}\Q\;/\;M_{e_1}^{-1}(a+j_3b)M_{e_1}(\cS\times
(I_1\otimes_R\cS))\subseteq\cS\times (I_1\otimes_R\cS)\}.
\]
We obtain that
$M_{e_1}^{-1}(a+j_3b)M_{e_1}=a+M_{e_1}^{-1}j_3M_{e_1}b$ where:
\[
M_{e_1}^{-1}j_3M_{e_1}=\left(\begin{array}{cc}-j_2 & 0\\0 &
j_2\end{array}\right)\frac{1}{e_{1,2}^\sigma
}\left(\begin{array}{cc}e_{1,1}^\sigma & \Norm(e_1)\\1 &
e_{1,1}\end{array}\right).
\]
Hence the $R$-module $\Lambda$ consists of elements $a+j_3b\in \Lambda^0$ with $a,b\in K$ such that,
for all $x\in\cS$ and all $y\in(I_1\otimes_R\cS)$,
\[
\left(\begin{array}{c}ax \\
ay\end{array}\right)+\left(\begin{array}{cc}-j_2 & 0\\0 &
j_2\end{array}\right)\frac{1}{e_{1,2}^\sigma
}\left(\begin{array}{cc}e_{1,1}^\sigma & \Norm(e_1)\\1 &
e_{1,1}\end{array}\right)\left(\begin{array}{c}bx \\
by\end{array}\right)\in \cS\times(I_1\otimes_R\cS).
\]
We deduce that
\begin{equation}\label{eq1}
\left\{\begin{array}{ll}
\displaystyle{ax-j_2\frac{e_{1,1}^\sigma bx+\Norm(e_1)by }{e_{1,2}^\sigma }} & \displaystyle{=ax+\frac{e_{2,1}(e_{1,1}^\sigma bx+\Norm(e_1)by) }{e_{1,2}^\sigma\cdot e_{2,2}}
-e_2\frac{e_{1,1}^\sigma bx+\Norm(e_1)by }{e_{1,2}^\sigma\cdot e_{2,2}}}\in\cS\\[8 pt]
\displaystyle{ay+j_2\frac{bx+e_{1,1}by}{e_{1,2}^\sigma}} & \displaystyle{=ay-\frac{e_{2,1}(bx+e_{1,1}by)}{e_{1,2}^\sigma\cdot
e_{2,2}}+e_2\frac{bx+e_{1,1}by}{e_{1,2}^\sigma\cdot
e_{2,2}}} \in I_1\otimes_R\cS.
\end{array}\right.
\end{equation}

Set $e_3=e_{2,1}\cdot e_{1,1}^\sigma-j_3e_{2,2}\cdot e_{1,2}^\sigma$. For all $a,b\in K$, we
have that $a+j_3b=a'+e_3b'$, where  $a'=a+\frac{e_{2,1}e_{1,1}^\sigma}{e_{2,2}\cdot
e_{1^2}^\sigma}b$ and $ b'=-\frac{b}{e_{2,2}\cdot e_{1,2}^\sigma}$.

Thus the expressions of \eqref{eq1} become (with this new basis $\langle 1,e_3\rangle$):
\begin{equation}\label{eq2}
(a'x-e_{2,1}\Norm(e_1)b'y) +e_2(e_{1,1}^\sigma b'x+\Norm(e_1)b'y )\in\cS
\end{equation}
\begin{equation}\label{eqp}
(a'y+e_{2,1}\Trace(e_{1,1})b'y+e_{2,1}b'x)-e_2(b'x+e_{1,1}b'y)\in I_1\otimes_R\cS.
\end{equation}

In particular, assuming $y=0$ we obtain from \eqref{eq2} that $(a'+e_2e_{1,1}^\sigma b')x\in\cS=R\oplus
e_2I_2$. This implies that $a'\in R$ and $e_{1,1}^\sigma b'\in I_2$. It follows from \eqref{eqp}
that $(e_{2,1}-e_2)b'x\in(I_1\otimes_R\cS)$, that is $b'^\sigma
j_2e_{2,2}=(e_{2,1}-e_2)b'\in(I_1\otimes_R\cS)$. Hence $b'\in I_1^\sigma I_2'$ where
\[
I_2'=\left\{\begin{array}{lr}I_2\cap \frac{1}{e_{2,1}}R & \mbox{if }e_{2,1}\neq 0\\I_2 & \mbox{if
} e_{2,1}=0.\end{array}\right.
\]

Assuming that $x=0$, it follows from \eqref{eq2} that $-(e_{2,1}-e_2)\Norm(e_1)b'y\in\cS$, which
is deduced from $(e_{2,1}-e_2)b'\in(I_1\otimes_R\cS)$ above and the fact that, since
$e_1I_1\in\cO$, $\Norm(e_1)I_1I_1^\sigma\subseteq R$. Moreover, by \eqref{eqp} we have that
$(a'+e_{2,1}\Trace(e_{1,1})b'-e_2e_{1,1}b')y\in I_1\otimes_R\cS$, which is again deduced from
$(e_{2,1}-e_2)b'\in I_1\otimes_R\cS$ and $(a'+e_2e_{1,1}^\sigma b')x\in\cS$, since
\[
(a'+e_{2,1}\Trace(e_{1,1})b'-e_2e_{1,1}b')y=y(a'+e_2e_{1,1}^\sigma
b')+(e_{2,1}-e_2)\Trace(e_{1,1}y)b'\in I_1\otimes_R\cS.
\]
for all $y\in I_1$ and $\Trace(e_{1,1}y)=\Trace(e_1y)\in\Trace(e_1I_1)\subseteq \Z\subset R$.

In conclusion, $a'+e_3b'\in \Lambda$ if and only if $a'\in R$ and
$b'\in I_3$, where
\begin{equation}\label{eqI3}
I_3=\left\{\begin{array}{lr}I_2'I_1^\sigma\cap \frac{1}{e_{1,1}^\sigma}I_2 & \mbox{if }e_{1,1}\neq
0\\I_2'I_1^\sigma & \mbox{if } e_{1,1}=0.\end{array}\right.
\end{equation}
Thus
$\Lambda\simeq R\oplus e_3I_3,$  where $e_3=e_{2,1}\cdot e_{1,1}-j_3e_{2,2}\cdot
e_{1,2}$,
and $j_3$ is a quaternionic complement of $\phi_s(P)$, such that $j_3^2=m_1m_2$.

\end{proof}

\section{Specialization of Heegner points with class number 1}\label{Esp_h_R1}

Let $p$ be a prime not dividing $D$. Notice that the special fiber $\cX_p$ at $p$ is a smooth curve over $\F_p$. We say that a point $P=[\tilde A,\tilde i]\in\cX_p(\overline\F_p)$ is \emph{supersingular} if $\tilde A$ is isogenous to a product of supersingular elliptic curves over $\overline\F_p$. Write $(\cX_p)_{ss}$ for the set of supersingular points of $\cX_p$.

It is well known that the set $(\cX_p)_{ss}$ is in one-to-one correspondence with $\Pic(Dp,1)$ (cf. \cite[\S 3]{Rib}). We denote the corresponding bijection by:
\begin{equation}\label{corrssing}
\varepsilon_{ss}:(\cX_p)_{ss}\stackrel{1:1}{\longleftrightarrow}\Pic(Dp,1).
\end{equation}
In analogy with the previous situation, for any $\tilde P=[\tilde A,\tilde i]\in(\cX_p)_{ss}$ the endomorphism ring $\End(\tilde A,\tilde i)$ is a maximal order in a quaternion algebra of discriminant $Dp$ endowed with a natural orientation (cf. \cite[Proposition 2.1]{Rib}). Moreover,
the map $\varepsilon_{ss}$ is given by $\varepsilon_{ss}(\tilde P)=\End(\tilde A,\tilde i)\in\Pic(Dp,1)$.

Let $K$ be an imaginary quadratic field and let $R$ be an order in $K$ of conductor $c$. Let $P=[A,i]\in\CM(R)$ be a Heegner point. Recall the description of the specialization map $\Pi:X_0^D(\overline\Q)\ra \cX_p(\F)$ of \eqref{espmap}. By \cite[\S 2.2]{Mol}, $A$ is isomorphic to the product of two isogenous elliptic curves with CM by $R$, say $A\simeq E_1\times E_2$. Therefore, since a CM elliptic curve specializes to a supersingular elliptic curve if and only if $p$ does not split in $K$, we deduce that $\Pi(P)=[\tilde A,\tilde i]\in(\cX_p)_{ss}$ if and only if $p$ does not split in $K$.

We proceed to describe the specialization of those Heegner points that lie in $\cX_p\setminus (\cX_p)_{ss}$.

\begin{proposition}\label{Ordred}
Let $P=[A,i]\in\CM(R)$ be a Heegner point and assume that $\Pi(P)=[\tilde A,\tilde i]\not\in(\cX_p)_{ss}$ (i.e. $p$ splits in $K$). Then the natural map
$\phi_P:\End^0(A,i)\hookrightarrow\End^0(\widetilde A,\widetilde i)$
is an isomorphism.
\end{proposition}
\begin{proof}
We have $A\cong E_1\times E_2$, where $E_1$ are $E_2$ are isogenous elliptic curves with CM by $R$. Write $\widetilde E_1$ and $\widetilde E_2$ for their specialization modulo $p$.
Since $[\tilde A,\tilde i]\not\in(\cX_p)_{ss}$, each curve $\widetilde E_i$ is an ordinary elliptic curve over $\overline\F_p$ such that $K=\End^0(E_i)=\End^0(\tilde E_i)$. This implies that $\End^0(\widetilde A)=M_2(K)=\End^0(A)$ and consequently $\phi_P(\End^0(A,i))=\End^0(\widetilde A,\widetilde i)$.
\end{proof}

In order to describe supersingular specialization, recall that, in case $P=[A,i]\in\CM(R)$ and $p$ does not split in $K$, the endomorphism ring $\End(\tilde A,\tilde i)$ acquires structure of oriented Eichler order in $\Pic(Dp,1)$. If in addition we assume that $c$ is prime-to-$p$, by \cite[Remark 4.2]{Mol} the natural monomorphism $\phi_P:\End(A,i)\ra\End(\tilde A,\tilde i)$ can be regarded as an optimal embedding in $\CM_{Dp,1}(R)$. One can see in \cite[\S 2.1]{Mol} that the set $\CM_{Dp,1}(R)$ is equipped with an action of the group $W(D)$ of Atkin-Lehner involutions. The following theorem relates the action of $W(D)$ on $P\in\CM(R)$ with the one on $\phi_P\in\CM_{Dp,1}(R)$.

\begin{theorem}\cite[Theorem 6.1]{Mol}\label{espss}
Let $P=[A,i]\in\CM(R)$. Assume that $p$ does not split in $K$ and $p\nmid c$. Then, the map
$P\mapsto(\End(A,i)\hookrightarrow\End(\tilde A,\tilde i))$ defines an injective map
    \begin{equation}\label{redss}
    \phi_{ss}:\CM(R)\longrightarrow\CM_{Dp,1}(R),
    \end{equation}
    satisfying $\phi_{ss}(\omega_n(P))=\omega_n(\phi_{ss}(P))$, for all $\omega_n\in W(D)$.
\end{theorem}

\begin{remark}\label{remss}
Recall the natural forgetful projection $\pi:\CM_{Dp,1}(R)\ra\Pic(Dp,1)$ defined in Remark \ref{remespsing}. Then, as in the previous setting, the specialization $\Pi(P)\in(\cX_p)_{ss}$ is determined by:
$$\varepsilon_{ss}(\Pi(P))=\pi(\phi_{ss}(P)).$$
\end{remark}

Assume from now on that $R$ has class number $h_R=1$. For any $P=[A,i]\in\CM(R)$, we proceed to compute the intersection index $(P,\omega_m (P))_p$ of \eqref{contact} for any $m\mid D$ in case of supersingular specialization.

Write $K_p=K\otimes_\Q\Q_p$, let $K_p^{\rm{unr}}$ be the maximal
unramified extension of $K_p$ and let $R_p^{\rm{unr}}$ be its
integer ring with uniformizer $\pi$. Write
$W_n=R_p^{\rm{unr}}/\pi^n R_p^{\rm{unr}}$. If $\omega_m
(P)=[A',i']\in\CM(R)$, we deduced in \S\ref{hypShi} that
$(\omega_m (P),P)_p=\max\{n:(A,i)\simeq(A',i')\;\mbox{over }W_n\}$.
The following theorem computes $(\omega_m (P),P)_p$ explicitly.

\begin{theorem}\label{compd}
Let $P=[A,i]\in\CM(R)$ be a Heegner point and let $p\nmid D$ be a prime that does not split
in $K$ and does not divide the conductor of $R$. Let
$\Lambda=\End(\tilde A,\tilde i)\in\Pic(Dp,1)$ and write
$\Lambda^0=\Lambda^0_++\Lambda^0_-$, where
$\Lambda_+^0=\phi_{ss}(P)(K)$ and $\Lambda_-^0$ is its quaternionic
complement. Let $\Lambda\simeq R\oplus eR$ be the decomposition of
$\Lambda$ provided by its free right $R$-module structure via
$\phi_{ss}(P)$. For any $\lambda\in\Lambda^0$, write
$\lambda=\lambda_++\lambda_-$, where $\lambda_+\in\Lambda^0_+$ and
$\lambda_-\in\Lambda^0_-$. Finally, for any $\lambda\in\Lambda$,
write $\lambda=\lambda^++e\lambda^-$, where $\lambda^+,\lambda^-\in
R$. Then, the integer $(\omega_m (P),P)_p$ is given by:
\begin{equation}\label{ordcont}
(\omega_m(P),P)_p=\max\left\{\frac{\ord_p(\Norm(\lambda^-))}{1-\left(\frac{d}{p}\right)}+1\;:\;\lambda\in\Lambda,\;\Norm(\lambda)=m\right\}
\end{equation}
where $d$ is the discriminant of $K$ and $\left(\frac{d}{p}\right)$
is the usual Legendre symbol. Moreover, if $\lambda\in\Lambda$ is
such that $\Norm(\lambda)=m$, the following equality holds:
\begin{equation}\label{dioph}
-dc^2m=-dc^2\Norm(\lambda_+)+\Norm(\lambda^-)D\cdot p.
\end{equation}
\end{theorem}
\begin{proof}
Since $h_R=1$, there is a single isomorphism class of elliptic curves $E$ with CM by $R$, and $E$ has supersingular specialization modulo $p$. Due to the fact that $E$ has potentially good reduction, after extending $K_p^{\rm{unr}}$ if necessary, we can choose a smooth model $\E$ of $E$ over $R_p^{\rm{unr}}$.

Denote by $\cS^n=\End_{W_n}(\E)$. In particular, $\cS^1=\cS=\End(\tilde E)$ shall be regarded as an element of $\Pic(p,1)$. The monomorphism of algebras $\phi:K\simeq\End^0(E)\hookrightarrow\End^0(\tilde E)$ yields a decomposition
$$\cS^0=\cS^0_+\oplus\cS^0_-,$$
where $\cS^0_+=\phi(K)$ and $\cS^0_-$ is its quaternionic complement. Then by the work of Gross-Zagier \cite[Proposition 3.7.3]{GrossZagier},
$$\cS^n=\End_{W_n}(\E)=\{\alpha\in\cS\;:\;d\cdot \Norm(\alpha_-)\equiv 0\mod p\cdot\Norm(\dP)^{n-1}\},$$
where $\dP\subset R$ is the prime ideal lying above $p$.

By \cite[\S 2.2]{Mol}, the abelian surface $A$ is isomorphic to $E^2$. Hence, in order to specialize $(A,i)$ over $W_n$ as in the above setting, we must consider a smooth model $\mathcal{A}$ of $A$ over $R_p^{\rm{unr}}$ and reduce modulo $\pi^n$. Write
$$\Lambda^n:=\End_{W_n}(A,i)=\{\lambda\in\End_{W_n}(\mathcal{A})\;:\;i(\alpha)\lambda=\lambda i(\alpha)\;\forall \alpha\in \cO\}.$$
We claim that:
\[
\Lambda^n=\{\alpha\in\Lambda\;:\;d\cdot \Norm(\alpha_-)\equiv 0\mod p\cdot\Norm(\dP)^{n-1}\}.
\]
Indeed, since $A\simeq E^2$, we have $\End_{W_n}(\mathcal{A})\simeq M_2(\cS^n)$. Moreover, due to the fact that $\cO\simeq R\times R$ as a right $R$-module via $\varphi(P)$, the $(\cO,\cS^n)$-bimodule $\cM_{P}^n=\cO\otimes_{R}\cS^n$ is isomorphic to $\cS^n\times\cS^n$ as a right $\cS^n$-order. By a similar argument as in \cite[Theorem 4.2]{Mol}, we obtain that $\Lambda^n=\End_\cO^{\cS^n}(\cM_{P}^n)$.

For any prime $q\neq p$ $\cS_q=\cS^n_q$. Hence we deduce that
$\Lambda^n_q=\Lambda^1_q=\Lambda_q$. On the other hand,
$\Lambda_p^n$ corresponds to matrices in $M_2(\cS_p^n)$ that commute
with $\cO_p=M_2(\Z_p)$, thus $\Lambda_p^n\simeq \cS_p^n$. Applying
the description of $\cS^n$ above, the desired claim follows.

Let us consider $\omega_m (P)=[A',i']\in\CM(R)$. By \cite[Corollary 2]{M-G}, the abelian surfaces with QM $(A,i)$ and $(A',i')$ are isogenous.
By this we mean that there exists an isogeny $\lambda:A\ra A'$ making, for all $\alpha\in\cO$, the following diagram commutative:
\[
\xymatrix{
A \ar[d]_{i(\alpha)}\ar[r]^\lambda & A'\ar[d]^{i'(\alpha)}\\
A  \ar[r]^\lambda & A'
}
\]

Write $I_m^n=\Hom_{W_n}((A,i),(A',i'))$ for the set of isogenies between $(A,i)/W_n$ and $(A',i')/W_n$. Then it is easy to check that $I_m^n=\Hom_\cO^{\cS_n}(\cM_{P}^n,\cM_{\omega_m(P)}^n)$ and, consequently, $I_m^n$ is a right $\Lambda^n$-module. Clearly, $(A,i)\simeq (A',i')$ over $W_n$ if and only if $I_m^n$ is principal.

By \cite[Remark 4.10]{Mol}, $I_m^1=\Hom_{\F}((A,i),(A',i'))$ is the
single (two-sided) ideal of $\Lambda$ of norm $m$. Hence, if $(A,i)$
and $(A',i')$ were isomorphic over $\F$, $I_m^1$ would be principal,
i.e. it would be generated by an element $\lambda\in\Lambda$ of
norm $m$.

Since $p\nmid D$, under the embedding
$\Lambda^n\hookrightarrow\Lambda$, the ideal $I_m^n$ is the only
ideal in $\Lambda^n$ lying above $I_m^1$. This means that $I_m^n$ is
principal if and only if there exists an element
$\lambda\in\Lambda_n$ that generates $I_m^1$, or equivalently,
$\lambda\in\Lambda$, $\Norm(\lambda)=m$ and $d\Norm(\lambda_-)\equiv
0\mod p\cdot\Norm(\dP)^{n-1}$. Computing the norm $\Norm(\dP)$ in
both cases $p\mid d$ and $p\nmid d$ we conclude that :
\begin{equation}\label{eqaux}
(\omega_m(P),P)_p=\left\{\begin{array}{lc}
\max\{\ord_p(d\cdot \Norm(\lambda_-))\;:\;\lambda\in\Lambda,\;\Norm(\lambda)=m\} & p\mid d\\
\max\{\frac{1}{2}(\ord_p(\Norm(\lambda_-))+1)\;:\;\lambda\in\Lambda,\;\Norm(\lambda)=m\}
& p\nmid d
\end{array}\right.
\end{equation}

Finally, the decomposition $\Lambda\simeq R\oplus eR$, where $e=e_++e_-$, allows us to compute the reduced discriminant of $\Lambda$ in terms of $R$ and $e$. Indeed we obtain that
$\disc(\Lambda)=e_-^2c^2 d$. Since $\Lambda\in\Pic(Dp,1)$ we deduce $Dp=e_-^2c^2 d$ (Notice that $d<0$ and $e_-^2<0$ since $\Lambda^0=\left(\frac{d,e_-^2}{\Q}\right)$ is definite).

For any $\lambda\in\Lambda$, we have that
$\lambda_+=\lambda^++\lambda^-\cdot e_+$ and $\lambda_-=e_-\cdot
\lambda^-$. If in addition $\Norm(\lambda)=m$, then
$\Norm(\lambda)=\Norm(\lambda_+)+\Norm(\lambda_-)=m$, where
$\Norm(\lambda_-)=-e_-^2\cdot\Norm(\lambda^-)=-\frac{\Norm(\lambda^-)Dp}{c^2d}$. Thus,
\[
-dc^2m=-dc^2\Norm(\lambda_+)+\Norm(\lambda^-)D p.
\]

Since by hypothesis $\ord_p(c)=0$, we have that $$\ord_p(d\cdot
\Norm(\lambda_-))=\ord_p(dc^2\cdot \Norm(\lambda_-))=\ord_p(-pD
\cdot\Norm(\lambda^-))=\ord_p(\Norm(\lambda^-))+1.$$
Finally, one obtains the desired formula from \eqref{eqaux}.
\end{proof}

\begin{remark}\label{remfinnum}
Notice that the integers $-dc^2m,-dc^2\Norm(\lambda_+)$ and
$\Norm(\lambda_-)D p$ are all positive. Hence, given $D$, $m$ and
$d$, equation \eqref{dioph} gives a finite number of possible $p$
and $\Norm(\lambda_-)$. Moreover, the valuation of such
$\Norm(\lambda_-)$ at $p$ provides the intersection index
$(\omega_m(P),P)_p$.
\end{remark}

\section{Algorithm to compute equations}\label{Algorithm}

Let $X=X_0^D/\Q$ be an hyperelliptic Shimura curve of genus $g\geq
3$ and let $\cX/\Z$ be Morita's integral model of $X$. Assume that we can
obtain a Weierstrass model $\cW$ of $X$ by blowing down certain
exceptional divisors of some special fibers of $\cX$. We proceed to describe an
algorithm to compute an hyperelliptic equation for $\cW$ over $\Z[1/2]$:
\[
\cW:y^2=R(x), \quad R(x)\in\Z[x],\quad\deg(R)= 2g+2.
\]

\subsection*{Step 1: Reduction of the set of Weierstrass points at bad primes}


Let $\CM(R_i)$ be a set of Heegner points in $\WP(X)$.
By \cite[Theorem 5.11 and Theorem 3.1]{Vig},
$\CM_{D,1}(R_i)$ is a $\Pic(R_i)$-orbit. Thus by
Theorem \ref{theoryCM}, the set $\CM(R_i)$ is a Galois
orbit. The decomposition $\WP(\cW)=\bigsqcup_j\CM(R_j)$ gives rise to a
factorization $R(x)=\prod_j p_{R_j}(x)$, where each $p_{R_i}\in\Z[x]$ is
irreducible, $\deg(p_{R_i})=\#\Pic(R_i)$ and roots of $p_{R_i}$
correspond to Weierstrass points $\CM(R_i)$. Moreover, the splitting field
of each $p_{R_i}$ coincides with the field of definition of any
$P\in\CM(R_i)$.

Fix $P\in\CM(R_i)$ and
let $p\mid D$ be a prime. Since $R_i^0=\Q(\sqrt{-D})$, Theorem \ref{espsing} asserts that its specialization $\Pi(P)$ lies in the singular locus $(\cX_p)_{\rm{sing}}$.
By Remark \ref{remespsing}, we are able to compute $\Pi(P)$ through the map $\phi_s$ of Theorem \ref{espsing}. Indeed, upon the correspondence $\varepsilon_s$ of \eqref{corrsing}:
\[
\varepsilon_s(\Pi(P))=\pi(\phi_s(P))\in\Pic(D/p,p).
\]
Finally, in order to compute $\phi_s$ we exploit the theory of bimodules and the algebraic description of $\phi_s$. In fact, Theorem \ref{compphi} gives $\phi_s(P)$ explicitly.

Once we have obtained $\varepsilon_s(\Pi(P))$ for a fixed $P\in\CM(R_i)$, we proceed to obtain the specialization of all $Q\in\CM(R_i)$ using the fact that $\CM(R_i)$ is a Galois orbit.
By Theorem \ref{espsing},
\begin{equation}\label{redhegpt}
\varepsilon_s(\Pi(P^{\Phi_{R_i}([J])}))=\pi(\phi_s(P^{\Phi_{R_i}([J])}))=\pi([J]\ast\phi_s(P)).
\end{equation}
Moreover, since we have an explicit description of $\phi_s(P)$ and the $\Pic(R)$-action on $\phi_s(P)$ is easily computable with \emph{MAGMA} \cite{Mag}, we obtain the specialization of all points in $\CM(R_i)$.

Notice that this recipe provides $\varepsilon_s(\Pi(Q))\in\Pic(D/p,Np)$ for all $Q\in\bigsqcup\CM(R_i)$ which, by Lemma \ref{espord}, describes its specialization and its thickness in $\cX_p$. In order to obtain its thickness in $\cW_p$ we apply formula \eqref{thickblow_d}.

\subsection*{Step 2: Choice of the points at infinity}\label{leadT}

As pointed out in \S \ref{hypShi}, we may choose an order $R_\infty$ with class number $h_{R_\infty}=1$ in an imaginary quadratic field $K_\infty$ of discriminant prime-to-$D$, such that $\CM(R_\infty)\neq\emptyset$. Notice that we can always assume that $R_\infty$ is maximal.
Fix $P_\infty=[A_\infty,i_\infty]\in\CM(R_\infty)$ and assume that $\{P_\infty,\omega_D(P_\infty)\}$ are the points at infinity. This fixes an affine open set of $\cW$ defined, over $\Z[1/2]$, by the equation $y^2=R(x)=\prod p_{R_i}(x)$, where $\deg(R(x))=2g+2$ and the factorization $R(x)=\prod p_{R_i}(x)$ is attached to the decomposition $\WP(\cW)=\bigsqcup_i \CM(R_i)$. Let $a_R$ and $a_{R_i}$ be the leading coefficients of $R(x)$ and $p_{R_i}(x)$ respectively, $a_R=\prod_i a_{R_i}$. Since $\Q(\sqrt{a_R})=K_\infty$, we control the sign of $a_R$ (which is negative since $K_\infty$ is imaginary) and its absolute value modulo squares.

In order to determine $a_R$, recall that $(a_R,D)=1$ and primes dividing $a_R$ correspond to places where both points at infinity specialize to the same $\F_p$-rational Weierstrass point. Thus, $\Pi(P_\infty)=\Pi(\omega_D (P_\infty))=\Pi(P)=[\tilde A,\tilde i]$ for some $P=[A,i]\in\WP(\cW)$.
Suppose that $\Pi(P_\infty)=\Pi(P)\not\in(\cX_p)_{ss}$. Then, by Proposition \ref{Ordred}, $K\simeq\End^0(A,i)\simeq\End^0(\tilde A,\tilde i)\simeq\End^0(A_\infty,i_\infty)\simeq K_\infty$, which is impossible since discriminant of $K_\infty$ is prime-to-$D$. Hence, for all primes $p\mid a_R$, $\Pi(P_\infty)=\Pi(\omega_D(P_\infty))\in(\cX_p)_{ss}$; equivalently, $p$ does not split in both $K$ and $K_\infty$.

Assume that $p\mid a_R$. By relation \eqref{contact}, the valuation
of $a_R$ at $p$ is given by $$\nu_p(a_R)=\left(1-\left(\frac{d}{p}\right)\right)(P_\infty,\omega_D
(P_\infty))_p.$$
Since $\Pi(P_\infty)\in(\cX_p)_{ss}$, we deduce from
Theorem \ref{compd} that:
\[
\nu_p(a_R)=
\max\left\{\ord_p(\Norm(\lambda^-))+1-\left(\frac{d}{p}\right)\;:\;\lambda\in\varepsilon(\Pi(P_\infty)),\;\Norm(\gamma)=D\right\}
\]
where $d$ is the discriminant of $K_\infty$. Moreover, for any
$\lambda\in\Lambda$ such that $\Norm(\lambda)=D$, the following
relation holds:
\[
-dD=-d\Norm(\lambda_+)+\Norm(\lambda^-)D p.
\]
This gives a finite number of possible $p$ and $\Norm(\lambda^-)$
for given $D$ and $d=\disc(K_\infty)$. Consequently, we have a
finite number of possible $\nu_p(a_R)$.

Once we have the set of possible $p$ dividing $a_R$, in order to determine which $a_{R_i}$
is divisible by $p$ recall the maps $\phi_{ss}$ of \eqref{redss}
attached to supersingular specialization. By Remark \ref{remss},
$p\neq 2$ divides $a_{R_i}$ if and only if
$\varepsilon(\Pi(P_\infty))=\pi(\phi_{ss}(P_\infty))\in\pi(\phi_{ss}(\CM(R_i)))$. Equivalently, $R_i$ is embedded in $\varepsilon(\Pi(P_\infty))\in\Pic(Dp,1)$ optimally.
There exists no pair of orders $R_i\neq R_j$ embedding optimally
in the same $\Lambda\in\Pic(Dp,1)$ since $\phi_{ss}$ is injective and two
Weierstrass points can not have the same specialization whenever $p$ is a prime of good reduction.

We are able to compute
$\varepsilon(\Pi(P_\infty))=R_\infty\oplus eR_\infty$, and
consequently we shall check whether $R_i$ is embedded optimally in
it.

In case $p=2$, we control the valuation $\nu_2(a_R)$ but we do not control the 2-valuation of each $a_{R_i}$ if $\nu_2(a_R)\neq 0$. In any case we have an upper bound; $\nu_2(a_{R_i})\leq\nu_2(a_{R})$.

\subsection*{Step 3: Discriminants, Resultants and Fields of definition}

For any $P\in\WP(\cW)$, write $\gamma_P$ for the root of $R(x)$ attached to $P$. Since we control the specialization of every point in $\WP(\cW)$ and we know how to compute its thickness, Theorem \ref{EspSing} yields the valuations $\nu_p(\gamma_{P}-\gamma_{P'})$ for every $P,P'\in\WP(\cW)$ and every $p\neq 2$.
This provides the discriminants  $\disc(p_{R_i})$ and the resultants $\Res(p_{R_i},p_{R_j})$ up to a power-of-$2$ factor,  namely
\begin{equation}\label{discarrels}
\nu_p(\disc(p_{R_i}))=\sum_{P,P'\in\CM(R_i)}2\cdot\nu_p(\gamma_{P}-\gamma_{P'}),\quad\nu_p(\Res(p_{R_i},p_{R_j}))
=\sum_{\substack{P\in\CM(R_i)\\
Q\in\CM(R_j)
}}\nu_p(\gamma_{P}-\gamma_{Q}).
\end{equation}
If in addition we assume good reduction at $2$, by \eqref{Disc} we have that
\begin{equation}\label{pow2}
4(g+1)=\nu_2(\disc(R))=\sum_i\nu_2(\disc(p_{R_i}))+\sum_{i,j}\nu_2(\Res(p_{R_j},p_{R_j})^2).
\end{equation}
In general we obtain a finite number of possible powers of $2$ dividing $\disc(p_{R_i})$ and $\Res(p_{R_i},p_{R_j})$.

By Theorem \ref{theoryCM}, points in $\CM(R_i)$ are defined over a subfield of the ring class field $H_{R_i}$ of $R_i$. We compute such field using the following theorem:
\begin{theorem}\cite[Theorem 5.12]{GoRo1}\label{subRCF}
Let $Q\in\CM(R)\subset X_0^D(H_R)$ for some order $R$ in the imaginary
quadratic field $K=\Q(\sqrt{-D})$. Fix an embedding $H_R\subset\C$ and denote by $c$ the complex conjugation. Then $[H_R:\Q(Q)]=2$ and
$\Q(Q)\subset H_R$ is the subfield fixed by
$\sigma=c\cdot\Phi_{R}([\mathfrak{a}])\in\Gal(H_R/\Q)$ for some
ideal $\mathfrak{a}$ such that
$B\simeq\left(\frac{-D,\Norm_{K/\Q}(\mathfrak{a})}{\Q}\right)$.
\end{theorem}

\begin{remark}\label{remalgquat}
One can see in \cite[\S 5]{GoRo1} that the class
$[\mathfrak{a}]\in\Pic(R)$ does depend on $Q$. Assume that
$[\mathfrak{a}]=[\mathfrak{c}]^2[\mathfrak{b}]$. Then, the Heegner
point $P=\Phi_R([\mathfrak{c}])(Q)\in\CM(R)$ is fixed by
$c\cdot\Phi_{R}([\mathfrak{b}])$, indeed
\[
c\cdot\Phi_R([\mathfrak{b}])(P)=c\cdot\Phi_R([\mathfrak{b}][\mathfrak{c}])(Q)=c\cdot\Phi_R([\mathfrak{c}]^{-1}[\mathfrak{a}])(Q)=\Phi_R([\mathfrak{c}])\cdot
c\cdot\Phi_R([\mathfrak{a}])(Q)=P
\]
Thus, for any $\mathfrak{b}$ verifying that $[\mathfrak{a}]\cdot[\mathfrak{b}]^{-1}\in\Pic(R)^2$, there exists some
$P\in\CM(R)$ such that $\Q(P)$ is the fixed field by
$c\cdot\Phi_{R}([\mathfrak{b}])$.

Let $M_{R_i}$ be the isomorphism class of the field $\Q(P)$, for any $P\in\CM(R_i)$. Then $M_{R_i}$ is characterized by the class $\{\mathfrak{a}\}\in\Pic(R)/\Pic(R)^2$. It is clear that any ideal $\mathfrak{b}$ in
$\{\mathfrak{a}\}$ satisfies the
isomorphism
$B\simeq\left(\frac{-D,\Norm_{K/\Q}(\mathfrak{b})}{\Q}\right)$.
In general, the converse is not true, but if $[H_R:  H]$ is odd, where $H$ is the Hilbert class field of $K$, then $\{\mathfrak{a}\} \in \Pic(R)/\Pic(R)^2$ is uniquely determined by such isomorphism (see
 \cite[Remark 5.11]{GoRo1}). In our particular setting, the conductor of $R$ is $2$ and, thus, $[H_R:  H]$ is either $1$ or $3$.
\end{remark}

This results yields the field $M_{R_i}$ attached to $\CM(R_i)$. Recall that this field coincides
with the splitting field of $p_{R_i}(x)$.

\subsection*{Step 4: Computing equations}

Since we have computed the leading coefficients of each $p_{R_i}$, we are able to convert them into monic polynomials.
Given $p_{R_i}(x)\in\Z[x]$ of discriminant $d$, leading coefficient $a_{R_i}$ and degree $n$, the polynomial $q_{R_i}(x)=a_{R_i}^{n-1}(p_{R_i}(x/a_{R_i}))$ turns out to be monic with integer coefficients and discriminant $a_{R_i}^{2n-2}d$. It defines the same field as $p_{R_i}(x)$.

Let $\delta_{R_i}$ be any root of $q_{R_i}$. Since $q_{R_i}\in\Z[x]$ is monic, the root $\delta_{R_i}$ belongs to $\cO_{M_{R_i}}$, the ring of integers of $M_{R_i}$.
Moreover, $\disc(q_{R_i})$ provides the $\Z$-index $[\cO_{M_{R_i}}:\Z[\delta_{R_i}]]$.
Through the instruction \emph{IndexFormEquation} of \emph{MAGMA} \cite{Mag} we obtain all possible $\delta_{R_i}$ of given index, up to sign and translations by integers. Thus, we are able compute all possible polynomials $q_{R_i}$ (and consequently $p_{R_i}$) up to transformations of the form $p(x)\rightarrow p(\pm x+r)$ with $r\in\Z$.
The polynomials $p_{R_i}$ can be determined with no ambiguity by means of the resultants $R_{i,j}=\Res(p_{R_i},p_{R_j})$.
Namely, given $p_{R_i}(x+r_i)$ and $p_{R_j}(x+r_j)$, the equation $R_{i,j}=\Res(p_{R_i}(x+r_i),p_{R_j}(x+r_j))$ provides the difference $r_i-r_j$. This way we obtain the product $p_{R_i}\cdot p_{R_j}$ up to translations by an integer.
Notice that, given the equation $y^2=R(x)$, the polynomial $R(x)$ is also defined up to translations by an integer.

\section{Siksek-Skorogatov Shimura curve $D=3\cdot 13$}\label{S-S exmpl}

In this section we shall compute an explicit equation for the hyperelliptic Shimura curve of discriminant $D=39$ exploiting the algorithm explained above. This curve was used in \cite{S-S} by Siksek and Skorogatov in order to find a counterexample to the Hasse principle explained by the Manin obstruction. Since their results depend on the conjectural equation of the curve given by Kurihara \cite{Kur}, the verification of such conjectural equation shows that the results of \cite{S-S} are unconditionally true.

\subsection*{Step 1: Reduction of the set of Weierstrass points at bad primes}

Let $X$ denote the hyperelliptic Shimura curve $X_0^{39}/\Q$. By Proposition \ref{WP}, $\WP(X)=\CM(R)\bigsqcup\CM(R_0)$, where $R_0=\Z[\frac{1+\sqrt{-39}}{2}]$ and $R=\Z[\sqrt{-39}]$. Let $K=\Q(\sqrt{-39})$. Notice that both $R$ and $R_0$ have class number 4, so their ring class fields have degree 4 over $K$.

We can compute the geometric special fiber of $\cX$ at 3 and 13
by means of Cerednik-Drinfeld's theory (cf. \cite[\S 3]{KoRo}
for a step-by-step guide on the computation of these special fibers using
\emph{MAGMA} \cite{Mag}). Notice that, in this case, $\cX=\cW$ since
$\cX/\langle\omega_D\rangle$ is smooth over $\Z$. In the drawings below, the integer
on each singular point stands for its thickness:

\begin{picture}(250,70)
\thicklines\qbezier(70,20)(90,-10)(110,20)

\qbezier(70,40)(90,70)(110,40)

\qbezier(110,40)(140,-20)(170,40)

\qbezier(110,20)(140,70)(170,20)

\qbezier(10,40)(40,-20)(70,40)

\qbezier(10,20)(40,70)(70,20)

\qbezier(260,20)(280,-10)(300,20)

\qbezier(260,40)(280,70)(300,40)

\qbezier(300,40)(330,-20)(360,40)

\qbezier(300,20)(330,70)(360,20)

\qbezier(200,40)(230,-20)(260,40)

\qbezier(200,20)(230,70)(260,20)

\put(16,28){\circle*{6}}

\put(65,28){\circle*{6}}

\put(116,28){\circle*{6}}

\put(165,28){\circle*{6}}

\put(206,28){\circle*{6}}

\put(255,28){\circle*{6}}

\put(306,28){\circle*{6}}

\put(355,28){\circle*{6}}

\put(16,38){1}

\put(65,38){1}

\put(116,38){1}

\put(165,38){1}

\put(206,38){2}

\put(255,38){3}

\put(306,38){2}

\put(355,38){1}

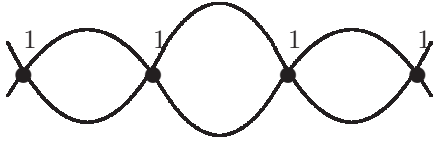
\put(40,-10){\mbox{Special fiber at }p=3}

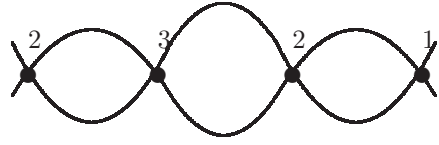
\put(230,-10){\mbox{Special fiber at }p=13}

\end{picture}

\quad

\quad

Let $\cO$ be a maximal order in the quaternion algebra $B$ of discriminant $39$. Choose arbitrary points $P\in\CM(R)$ and $P_0\in\CM(R_0)$. As it is more convenient for computations to work with optimal embeddings instead of Heegner points, let $\varphi(P)\in\CM_{39,1}(R)$ and $\varphi(P_0)\in\CM_{39,1}(R_0)$ be the optimal embeddings attached to $P$ and $P_0$, respectively, via $\eqref{Phi}$. In particular, $\varphi(P)$ and $\varphi(P_0)$ yield the following decompositions computed with \emph{MAGMA} \cite{Mag}:
\[
\left\{\begin{array}{c}
B=K\oplus i_1 K\\
\cO=R\oplus e_1I_1\\
\end{array}\right.\mbox{ where }\left\{\begin{array}{l}
i_1\mbox{ is a quaternionic complement of $\varphi(P)$,}\quad i_1^2=447\\
e_1=i_1+(7\cdot\sqrt{-39}+18)\\
I_1=\langle \frac{1}{2},\frac{\sqrt{-39}}{894}+\frac{63}{298} \rangle_R
\end{array}\right.
\]and
\[
\left\{\begin{array}{c}
B=K\oplus i_1' K\\
\cO=R_0\oplus e_1'I_1'
\end{array}\right. \mbox{ where }\left\{\begin{array}{l}
i_1'\mbox{ is a quaternionic complement of $\varphi(P_0)$,}\quad i_1'^2=6\\
e_1'=i_1'\\
I_1'=\langle 1,\frac{\sqrt{-39}-9}{12} \rangle_{R_0}
\end{array}\right.
\]

\subsubsection*{Reduction modulo 3}

In order to compute the specialization modulo $p=3$ of $P$ and $P_0$, we shall compute the optimal embeddings $\psi_{R}\in\CM_{13,3}(R)$ and $\psi_{R_0}\in\CM_{13,3}(R_0)$ of Theorem \ref{redbim}. Their targets are maximal orders $\cS_3$ and $\cS_3'$ of the quaternion algebra $H_3$ of discriminant 3. Again both embeddings define the following decompositions:
\[
\left\{\begin{array}{c}
H_3=K\oplus i_2 K\\
\cS_3=R\oplus e_2I_2\\
\end{array}\right.\mbox{ where }\left\{\begin{array}{l}
i_2\mbox{ is a quaternionic complement, of $\psi_{R}$,}\quad i_2^2=-43\\
e_2=i_2-387\\
I_2=\langle \frac{1}{2},\frac{\sqrt{-39}}{1118}-\frac{10}{43} \rangle_R
\end{array}\right.
\]and
\[
\left\{\begin{array}{c}
H_3=K\oplus i_2' K\\
\cS_3'=R_0\oplus e_2'I_2'
\end{array}\right. \mbox{ where }\left\{\begin{array}{l}
i_2'\mbox{ is a quaternionic complement of $\psi_{R_0}$,}\quad i_2'^2=-12\\
e_2'=i_2'-12\\
I_2'=\langle 1,\frac{\sqrt{-39}-1}{156\cdot 2}-\frac{29}{78} \rangle_{R_0}
\end{array}\right.
\]

Hence, by Theorem \ref{compphi} the optimal embedding $\phi_{s}(P):R\hookrightarrow\End_\cO^{\cS_3}(\cO\otimes_R\cS_3)=\Lambda_3$ of \eqref{reds} is given by the decomposition:
\[
\left\{\begin{array}{c}
\Lambda_3\otimes\Q=K\oplus i_3 K\\
\Lambda_3=R\oplus e_3I_3\\
\end{array}\right.\mbox{ where }\left\{\begin{array}{l}
i_3\mbox{ is a quaternionic complement of $\phi_s(P)$,}\quad i_3^2=-43\cdot 447\\
e_3=-387\cdot (18-7\cdot\sqrt{-39})-i_3\\
I_3=(I_2\cap\frac{-1}{387}R)\overline{I_1}\cap\frac{1}{18-7\cdot\sqrt{-39}}I_2
\end{array}\right.
\]
Similarly $\phi_s(P_0):R_0\hookrightarrow\End_\cO^{\cS_3'}(\cO\otimes_{R_0}\cS_3')=\Lambda_3'$ is given by:
\[
\left\{\begin{array}{c}
\Lambda_3'\otimes\Q=K\oplus i_3' K\\
\Lambda_3'=R_0\oplus e_3'I_3'\\
\end{array}\right.\mbox{ where }\left\{\begin{array}{l}
i_3'\mbox{ is a quaternionic complement of $\phi_s(P_0)$,}\quad i_3'^2=-12\cdot 6\\
e_3'=-i_3'\\
I_3'=(I_2'\cap\frac{-1}{12}R_0)\overline{I_1'}.
\end{array}\right.
\]
Once we have a characterization of the embeddings $\phi_s(P)$ and $\phi_s(P_0)$, we proceed to describe the specialization of all Heegner points in $\CM(R)$ and $\CM(R_0)$. Recall that, in both cases, the sets $\CM(R)$ and $\CM(R_0)$ are $\Pic(R)$ and $\Pic(R_0)$-orbits respectively. Moreover, $\Pic(R)\simeq\Pic(R_0)\simeq\Z/4\Z$.

\emph{Case $\CM(R)$}: We pick a representative $J$ of a generator $[J]\in\Pic(R)$. We construct the left-$\Lambda_3$-ideals $\Lambda_3\phi_s(P)(J)$, $\Lambda_3\phi_s(P)(J^2)$, $\Lambda_3\phi_s(P)(J^3)$ and we compute their right orders $\pi([J^i]\ast\phi_s(P))$. We obtain that their number of units are:
$$\#(\Lambda_3^*)/2=\#(\pi([J]\ast\phi_s(P))^*)/2=\#(\pi([J^2]\ast\phi_s(P))^*)/2=\#(\pi([J^3]\ast\phi_s(P))^*)/2=1.$$
Thus, by \eqref{epsilon}, such integers are the thickness of each singular specializations.

Besides, we checked that $\Lambda_3\phi_s(P)(J)$ and $\Lambda_3\phi_s(P)(J^2)(\Lambda_3\phi_s(P)(J^3))^{-1}$ are principal, whereas $\Lambda_3\phi_s(P)(J^2)$, $\Lambda_3\phi_s(P)(J^3)$, $\Lambda_3\phi_s(P)(J)(\Lambda_3\phi_s(P)(J^2))^{-1}$, $\Lambda_3\phi_s(P)(J)(\Lambda_3\phi_s(P)(J^3))^{-1}$ are not.
Since for any pair of left $\Lambda_3$-ideals $I_1$ and $I_2$ their right orders are isomorphic as oriented Eichler orders if and only if $I_1\cdot I_2^{-1}$ is principal, it follows from \eqref{redhegpt} that
$$\Pi(P)=\varepsilon^{-1}(\pi(\phi_s(P)))=\varepsilon^{-1}(\pi([J]\ast\phi_s(P)))=\Pi(P^{\Phi_R([J])})$$ $$\Pi(P^{\Phi_R([J^2])})=\varepsilon^{-1}(\pi([J^2]\ast\phi_s(P)))=\varepsilon^{-1}(\pi([J^3]\ast\phi_s(P)))=\Pi(P^{\Phi_R([J^3])}).$$

\emph{Case $\CM(R_0)$}: Let $J'$ be a representative of a generator of $\Pic(R_0)$. Similarly as above, we construct the corresponding left-$\Lambda_3'$-ideals and we obtain: $$\#(\Lambda_3'^*)/2=\#(\pi([J']\ast\phi_s(P_0))^*)/2=\#(\pi([J'^2]\ast\phi_s(P_0))^*)/2=\#(\pi([J'^3]\ast\phi_s(P_0))^*)/2=1.$$

Moreover, we checked that $\Lambda_3'\phi_s(P_0)(J')$ and $\Lambda_3'\phi_s(P_0)(J'^2)(\Lambda_3'\phi_s(P_0)(J'^3))^{-1}$ are principal, whereas the remaining ones are not. Thus $\Pi(P_0)=\Pi(P_0^{\Phi_{R_0}([J'])})$  and $\Pi(P_0^{\Phi_{R_0}([J'^2])})=\Pi(P_0^{\Phi_{R_0}([J'^3])})$.

In conclusion we obtain the following diagram, describing the specialization of the Weierstrass points modulo $p=3$.

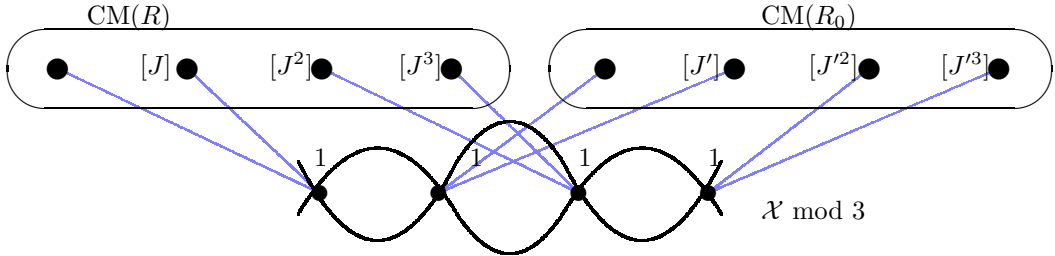
\begin{picture}(300,120)
\color[rgb]{0.5,0.5,1}

\qbezier(19,75)(118,28)(118,28)
\qbezier(68, 75)(118,28)(118,28)
\qbezier(119, 75)(216,28)(216,28)
\qbezier(168, 75)(216,28)(216,28)

\qbezier(226,75)(163,28)(163,28)
\qbezier(275, 75)(163,28)(163,28)
\qbezier(326,75)(265,28)(265,28)
\qbezier(375, 75)(265,28)(265,28)

\color{black}

\put(19, 75){\circle*{8}}

\put(68, 75){\circle*{8}}

\put(119,75){\circle*{8}}

\put(168,75){\circle*{8}}

\put(50, 73){$[J]$}

\put(99,73){$[J^2]$}

\put(148,73){$[J^3]$}

\put(226,75){\circle*{8}}

\put(275,75){\circle*{8}}

\put(326,75){\circle*{8}}

\put(375,75){\circle*{8}}

\put(255, 73){$[J']$}

\put(303,73){$[J'^2]$}

\put(353,73){$[J'^3]$}

\put(95, 75){\oval(190, 30)}

\put(300, 75){\oval(190, 30)}

\thicklines

\qbezier(170,20)(190,-10)(210,20)

\qbezier(170,40)(190,70)(210,40)

\qbezier(210,40)(240,-20)(270,40)

\qbezier(210,20)(240,70)(270,20)

\qbezier(110,40)(140,-20)(170,40)

\qbezier(110,20)(140,70)(170,20)

\put(118,28){\circle*{6}}

\put(163,28){\circle*{6}}

\put(216,28){\circle*{6}}

\put(265,28){\circle*{6}}

\put(285,18){$\cX$ mod 3}

\put(116,38){1}

\put(175,38){1}

\put(216,38){1}

\put(265,38){1}

\put(30,92){$\CM(R)$}

\put(285,92){$\CM(R_0)$}

\end{picture}

\quad

\subsubsection*{Reduction modulo 13}

With the same computations as in the previous setting, we obtain that the reduction of $\CM(R)$ and $\CM(R_0)$ modulo $p=13$ is given by the following diagram:

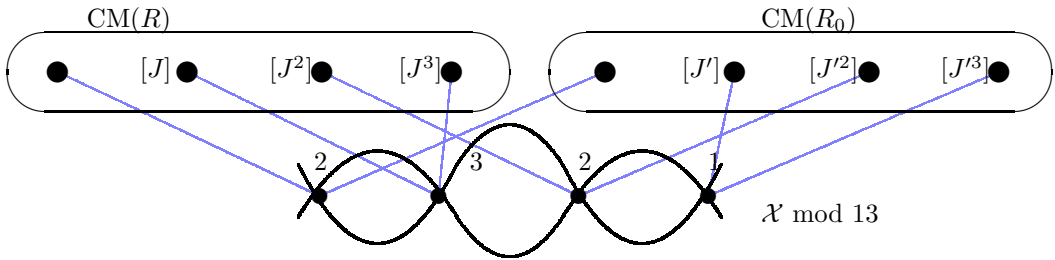
\begin{picture}(300,120)

\color[rgb]{0.5,0.5,1}

\qbezier(19,75)(118,28)(118,28)
\qbezier(68, 75)(163,28)(163,28)
\qbezier(119, 75)(216,28)(216,28)
\qbezier(168, 75)(163,28)(163,28)

\qbezier(226,75)(118,28)(118,28)
\qbezier(275, 75)(265,28)(265,28)
\qbezier(326,75)(216,28)(216,28)
\qbezier(375, 75)(265,28)(265,28)

\color{black}

\put(19, 75){\circle*{8}}

\put(68, 75){\circle*{8}}

\put(119,75){\circle*{8}}

\put(168,75){\circle*{8}}

\put(50, 73){$[J]$}

\put(99,73){$[J^2]$}

\put(148,73){$[J^3]$}

\put(226,75){\circle*{8}}

\put(275,75){\circle*{8}}

\put(326,75){\circle*{8}}

\put(375,75){\circle*{8}}

\put(255, 73){$[J']$}

\put(303,73){$[J'^2]$}

\put(353,73){$[J'^3]$}

\put(95, 75){\oval(190, 30)}

\put(300, 75){\oval(190, 30)}

\thicklines

\qbezier(170,20)(190,-10)(210,20)

\qbezier(170,40)(190,70)(210,40)

\qbezier(210,40)(240,-20)(270,40)

\qbezier(210,20)(240,70)(270,20)

\qbezier(110,40)(140,-20)(170,40)

\qbezier(110,20)(140,70)(170,20)

\put(118,28){\circle*{6}}

\put(163,28){\circle*{6}}

\put(216,28){\circle*{6}}

\put(265,28){\circle*{6}}

\put(285,18){$\cX$ mod 13}

\put(116,38){2}

\put(175,38){3}

\put(216,38){2}

\put(265,38){1}

\put(30,92){$\CM(R)$}

\put(285,92){$\CM(R_0)$}

\end{picture}

\quad

\subsection*{Step 2: Choice of the points at infinity}

Let $K_\infty=\Q(\sqrt{-7})$ and let
$R_\infty$ be its maximal order. As it is well known,
$\#\Pic(R_\infty)=1$. Hence, by \S \ref{leadT}, for any
$P_\infty\in\CM(R_\infty)$ we can choose $P_\infty$ and
$\omega_{39}(P_\infty)$ to be our points at infinity.
This choice of the points at infinity gives rise to an equation
\[
y^2=R(x),\quad\deg(R(x))=2g+2=8,
\]
defining the Weierstrass model $\cW$. Let
$R(x)=p_R(x)\cdot p_{R_0}(x)$ be the factorization
attached to the decomposition $\WP(\cW)=\CM(R)\sqcup\CM(R_0)$. Let
$a_R$ and $a_{R_0}$ be the leading coefficients of $p_R$ and $p_{R_0}$
respectively.

Since $\Q(\sqrt{a_R\cdot a_{R_0}})=K_\infty=\Q(\sqrt{-7})$, we deduce that $a_{R}\cdot a_{R_0}=-7\cdot N^2$ for some $N\in\Z$.
Given a prime $p$ dividing $a_R\cdot a_{R_0}$, by \eqref{dioph} we know that:
\[
7\cdot 39=m+\Norm(\lambda^-)39\cdot p,\mbox{  where
}m=7\Norm(\lambda_+)\in\Z^+.
\]
From this we obtain that $39\mid m=7\Norm(\lambda_+)\in
\Norm(K_\infty)$. Since $3$ and $13$ are inert in $K_\infty$, the
fact that $39\mid m\in\Norm(K_\infty)$ implies that $39^2\mid m$. Then,
dividing the above identity by 39, one obtains $7=39\cdot
m'+\Norm(\lambda^-) p$, where $m'\in\Z^+$. Thus $m=m'=0$, $p=7$ and
$\Norm(\lambda^-)=1$. Finally, by \eqref{ordcont} one
concludes that the leading coefficient of the hyperelliptic equation
must be $a_R\cdot a_{R_0}=-7$.

Moreover, we can compute $\pi(\phi_{ss}(\varphi(P_\infty)))\in\Pic(39\cdot 7,1)$ of Remark \ref{remss}. Namely, $$\pi(\phi_{ss}(\varphi(P_\infty)))=R_\infty\oplus jR_\infty,$$
where $jR_\infty$ is the quaternionic complement of $R_\infty$ with $j^2=-39$.
Since it can be checked that $R_0=\Z[\frac{1+\sqrt{-39}}{2}]$ can not be embedded in $\Lambda$, we conclude that $R=\Z[\sqrt{-39}]$ is embedded optimally in it. Therefore, $a_{R_0}=1$ and $a_R=-7$.

\subsection*{Step 3: Discriminants, Resultants and Fields of definition}

By Theorem \ref{subRCF}, points in $\CM(R)$ and $\CM(R_0)$ are
defined over a subfield of index 2 of the Hilbert class field $H_K$
of $K$. By Remark \ref{remalgquat}, to find such subextension we
must find an ideal $\mathfrak{a}$ of $R$ such that
$B\simeq\left(\frac{-39,\Norm_{K/\Q}(\mathfrak{a})}{\Q}\right)$. As one
checks, any $\mathfrak{a}$ such that
$\Norm_{K/\Q}(\mathfrak{a})=5$ does. Notice that $5$ splits in $K$, hence writing $5=\dP\cdot\dP'$ we have
$\Norm_{K/\Q}(\dP)=\Norm_{K/\Q}(\dP')=5$.

We used MAGMA \cite{Mag} to compute that the Hilbert class field of $K$ is defined by the polynomial $q(x)=x^4+4x^2-48$ over $K$. If $\alpha$ is any root of $q(x)$, then $H_K=\Q(\alpha,\sqrt{-39})$.

The automorphisms $\Phi_R(\dP)$ and complex conjugation $c$ act on $H_K$ by the rules:
\[
\Phi_R(\dP):\left\{\begin{array}{ccl}
\sqrt{-39}&\rightarrow &\sqrt{-39}\\
\alpha &\rightarrow & -\frac{\sqrt{-39}\alpha^3}{156} - \frac{7\sqrt{-39}\alpha}{39}
\end{array}\right.\quad c:\left\{\begin{array}{ccl}
\sqrt{-39}&\rightarrow &-\sqrt{-39}\\
\alpha &\rightarrow & -\alpha
\end{array}\right.
\]
Thus $\sigma=c\cdot\Phi_R(\dP)$ acts as:
\[
\sigma:\left\{\begin{array}{ccl}
\sqrt{-39}&\rightarrow &-\sqrt{-39}\\
\alpha &\rightarrow & -\frac{\sqrt{-39}\alpha^3}{156} - \frac{7\sqrt{-39}\alpha}{39}
\end{array}\right.
\]
We obtain that $M_R$, the fixed field by $\sigma$, is defined by the polynomial $x^4+8x^2-24x+16$ over $\Q$.
Since $\disc(M_R)=3^2\cdot 13$, we have that $\disc(p_R),\disc(p_{R_0})=N^2\cdot 3^2\cdot 13$, for certain $N\in\Z$.

Recall the following diagram summarizing the specialization of the Weierstrass points:

\begin{picture}(300,170)

\color[rgb]{0.5,0.5,1}

\qbezier(19,75)(100,110)(118,120)
\qbezier(68, 75)(100,100)(118,120)
\qbezier(119, 75)(216,120)(216,120)
\qbezier(168, 75)(216,120)(216,120)

\qbezier(226,75)(163,120)(163,120)
\qbezier(275, 75)(163,120)(163,120)
\qbezier(326,75)(265,120)(265,120)
\qbezier(375, 75)(265,120)(265,120)

\qbezier(19,75)(118,28)(118,28)
\qbezier(68, 75)(163,28)(163,28)
\qbezier(119, 75)(216,28)(216,28)
\qbezier(168, 75)(163,28)(163,28)

\qbezier(226,75)(118,28)(118,28)
\qbezier(275, 75)(265,28)(265,28)
\qbezier(326,75)(216,28)(216,28)
\qbezier(375, 75)(265,28)(265,28)

\color{black}

\put(19, 75){\circle*{8}}

\put(68, 75){\circle*{8}}

\put(119,75){\circle*{8}}

\put(168,75){\circle*{8}}

\put(226,75){\circle*{8}}

\put(275,75){\circle*{8}}

\put(326,75){\circle*{8}}

\put(375,75){\circle*{8}}

\put(95, 75){\oval(190, 30)}

\put(300, 75){\oval(190, 30)}

\thicklines

\qbezier(170,20)(190,-10)(210,20)

\qbezier(170,40)(190,70)(210,40)

\qbezier(210,40)(240,-20)(270,40)

\qbezier(210,20)(240,70)(270,20)

\qbezier(110,40)(140,-20)(170,40)

\qbezier(110,20)(140,70)(170,20)

\put(118,28){\circle*{6}}

\put(163,28){\circle*{6}}

\put(216,28){\circle*{6}}

\put(265,28){\circle*{6}}

\put(285,18){$\cX$ mod 13}

\put(116,38){2}

\put(175,38){3}

\put(216,38){2}

\put(265,38){1}

\qbezier(170,112)(190,82)(210,112)

\qbezier(170,132)(190,162)(210,132)

\qbezier(210,132)(240,72)(270,132)

\qbezier(210,112)(240,162)(270,112)

\qbezier(110,112)(140,162)(170,112)

\qbezier(110,132)(140,72)(170,132)

\put(30,92){$\CM(R)$}

\put(285,92){$\CM(R_0)$}

\put(118,120){\circle*{6}}

\put(163,120){\circle*{6}}

\put(216,120){\circle*{6}}

\put(265,120){\circle*{6}}

\put(285,130){$\cX$ mod 3}

\put(116,130){1}

\put(175,130){1}

\put(216,130){1}

\put(265,130){1}

\end{picture}

\quad

By Theorem \ref{EspSing} and \eqref{discarrels}, we have that $|\disc(p_R)|=2^{2k}\cdot3^2\cdot13^3$, $|\disc(p_{R_0})|=2^{2k'}\cdot3^2\cdot13$ and $\Res(p_R,p_{R_0})^2=2^{2k''}\cdot13^4$. Moreover, since $X_0^{39}$ has good reduction at 2,  \eqref{pow2} shows that $2k+2k'+2k''=16$.

\subsection*{Step 4: Computing equations}

Since the leading coefficient of $p_R$ is $a_{R}=-7$, we deduce that $q_R(x)=7^3p_R(x/7)$ is a monic polynomial of discriminant $7^{6}\disc(p_R)=2^{2k}\cdot3^2\cdot13^3\cdot 7^6$.

The instruction \emph{IndexFormEquation} of \emph{MAGMA} \cite{Mag} provides the possible candidates for $p_{R_0},$ $q_R$ and $p_R$ (denoted $\tilde p_{R_0}$, $\tilde q_R$ and $\tilde p_R$ respectively), up to transformations of the form $p(x)\rightarrow p(\pm x+r)$ with $r\in\Z$. We obtain that
\[
\tilde p_R(x)=\left\{
\begin{array}{cc}
-7x^4 - 51x^3 - 116x^2 - 84x - 19 & \disc(\tilde p_R)=3^2\cdot13^3\\
\left.\begin{array}{c}-7x^4 - 74x^3 - 200x^2 - 22x - 1\\-7x^4 + 38x^3 + 16x^2 - 182x - 169\end{array}\right\} & \disc(\tilde p_R)=2^{12}\cdot3^2\cdot13^3\\
\end{array}\right.
\]
and there are 16 more candidates $\tilde p_{R_0}(x)$ for $p_{R_0}(x)$, with discriminants $3\cdot13$, $2^4\cdot3\cdot13$, $2^{12}\cdot3\cdot13$ and $2^{16}\cdot3\cdot13$. If we compute the resultant $\Res(\tilde p_R(\mp x+\alpha),\tilde p_{R_0}(x))$ and look for solutions $\alpha\in\Z$ such that $\Res(\tilde p_R(\mp x+\alpha),\tilde p_{R_0}(x))^2=2^{2k''}\cdot13^4$, we obtain a single solution:
\[
p_{R_0}(x)=x^4 + 9x^3 + 29x^2 + 39x + 19,\quad
p_R(x)=-7x^4 - 79x^3 - 311x^2 - 497x - 277.
\]
In conclusion the equation we are looking for is
\[
\boxed{
y^2=-(7x^4 + 79x^3 + 311x^2 + 497x + 277)\cdot(x^4 + 9x^3 + 29x^2 + 39x + 19).
}
\]

Notice that this curve coincides with the one conjectured by Kurihara in \cite{Kur}.

\section{Case $D=5\cdot 11$}\label{Ex55}

Let $X$ be the hyperelliptic Shimura curve $X_0^{55}/\Q$. In this case the set of Weierstrass
points is $\WP(X)=\CM(\Z[\sqrt{-55}])\bigsqcup\CM(\Z[\frac{1+\sqrt{-55}}{2}])$ and both
$\Z[\frac{1+\sqrt{-55}}{2}]$ and $\Z[\sqrt{-55}]$ have class number 4.
As is the above situation, we can compute the geometric special fiber of $\cX$ at 5 and 11
using \cite[\S 3]{KoRo}.
In this case, the integral model $\cX$ does not correspond to a Weierstrass model
since $\cX/\langle\omega_D\rangle$ is not smooth over $\Z$.

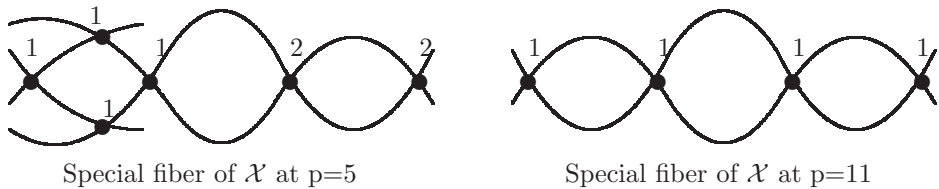
\begin{picture}(250,70)
\thicklines\qbezier(70,20)(90,-10)(110,20)

\qbezier(70,40)(90,70)(110,40)

\qbezier(110,40)(140,-20)(170,40)

\qbezier(110,20)(140,70)(170,20)

\qbezier(10,10)(40,-10)(70,40)

\qbezier(10,40)(30,10)(60,10)

\qbezier(10,20)(30,45)(60,50)

\qbezier(10,50)(40,60)(70,20)

\qbezier(260,20)(280,-10)(300,20)

\qbezier(260,40)(280,70)(300,40)

\qbezier(300,40)(330,-20)(360,40)

\qbezier(300,20)(330,70)(360,20)

\qbezier(200,40)(230,-20)(260,40)

\qbezier(200,20)(230,70)(260,20)

\put(18,28){\circle*{6}}

\put(63,28){\circle*{6}}

\put(116,28){\circle*{6}}

\put(45,45){\circle*{6}}

\put(45,11){\circle*{6}}

\put(165,28){\circle*{6}}

\put(206,28){\circle*{6}}

\put(255,28){\circle*{6}}

\put(306,28){\circle*{6}}

\put(355,28){\circle*{6}}

\put(16,38){1}

\put(65,38){1}

\put(40,50){1}

\put(45,15){1}

\put(116,38){2}

\put(165,38){2}

\put(206,38){1}

\put(255,38){1}

\put(306,38){1}

\put(353,38){1}

\put(30,-10){\mbox{Special fiber of $\cX$ at }p=5}

\put(220,-10){\mbox{Special fiber of $\cX$ at }p=11}

\end{picture}

\quad

\quad

In order to transform $\cX$ into a Weierstrass model $\cW$ we shall need to blow down the
exceptional divisors and apply relation \eqref{thickblow_d} to obtain new thicknesses.

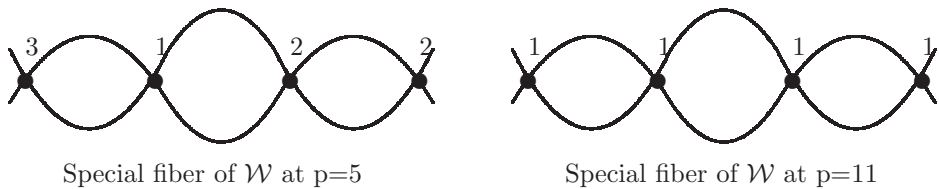
\begin{picture}(250,70)
\thicklines\qbezier(70,20)(90,-10)(110,20)

\qbezier(70,40)(90,70)(110,40)

\qbezier(110,40)(140,-20)(170,40)

\qbezier(110,20)(140,70)(170,20)

\qbezier(10,40)(40,-20)(70,40)

\qbezier(10,20)(40,70)(70,20)

\qbezier(260,20)(280,-10)(300,20)

\qbezier(260,40)(280,70)(300,40)

\qbezier(300,40)(330,-20)(360,40)

\qbezier(300,20)(330,70)(360,20)

\qbezier(200,40)(230,-20)(260,40)

\qbezier(200,20)(230,70)(260,20)

\put(16,28){\circle*{6}}

\put(65,28){\circle*{6}}

\put(116,28){\circle*{6}}

\put(165,28){\circle*{6}}

\put(206,28){\circle*{6}}

\put(255,28){\circle*{6}}

\put(306,28){\circle*{6}}

\put(355,28){\circle*{6}}

\put(16,38){3}

\put(65,38){1}

\put(116,38){2}

\put(165,38){2}

\put(206,38){1}

\put(255,38){1}

\put(306,38){1}

\put(355,38){1}

\put(30,-10){\mbox{Special fiber of $\cW$ at }p=5}

\put(220,-10){\mbox{Special fiber of $\cW$ at }p=11}

\end{picture}

\quad

\quad

Applying our algorithm, we obtain that the specialization of the Heegner points $\CM(\Z[\sqrt{-55}])$
and $\CM(\Z[\frac{1+\sqrt{-55}}{2}])$ in $\cX$ is given by the following diagram:

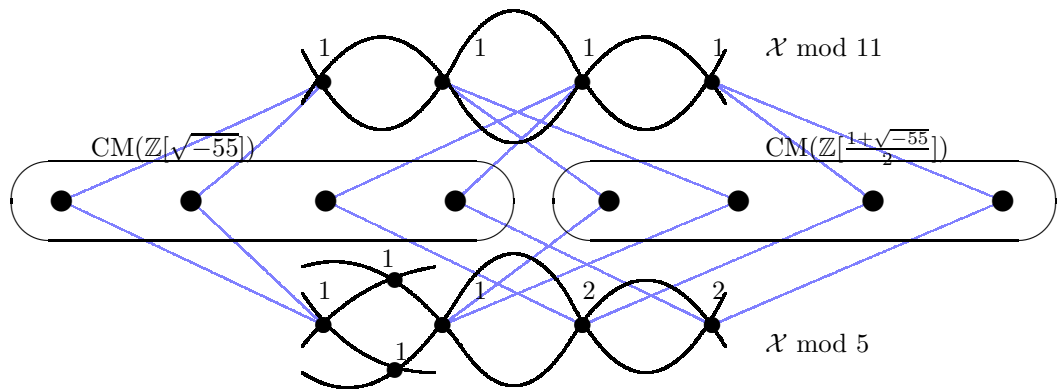
\begin{picture}(300,170)

\color[rgb]{0.5,0.5,1}

\qbezier(19,75)(100,110)(118,120)
\qbezier(68, 75)(100,100)(118,120)
\qbezier(119, 75)(216,120)(216,120)
\qbezier(168, 75)(216,120)(216,120)

\qbezier(226,75)(163,120)(163,120)
\qbezier(275, 75)(163,120)(163,120)
\qbezier(326,75)(265,120)(265,120)
\qbezier(375, 75)(265,120)(265,120)

\qbezier(19,75)(118,28)(118,28)
\qbezier(68, 75)(118,28)(118,28)
\qbezier(119, 75)(216,28)(216,28)
\qbezier(168, 75)(265,28)(265,28)

\qbezier(226,75)(163,28)(163,28)
\qbezier(275, 75)(163,28)(163,28)
\qbezier(326,75)(216,28)(216,28)
\qbezier(375, 75)(265,28)(265,28)

\color{black}

\put(118,28){\circle*{6}}

\put(163,28){\circle*{6}}

\put(216,28){\circle*{6}}

\put(265,28){\circle*{6}}

\put(19, 75){\circle*{8}}

\put(68, 75){\circle*{8}}

\put(119,75){\circle*{8}}

\put(168,75){\circle*{8}}

\put(226,75){\circle*{8}}

\put(275,75){\circle*{8}}

\put(326,75){\circle*{8}}

\put(375,75){\circle*{8}}

\put(95, 75){\oval(190, 30)}

\put(300, 75){\oval(190, 30)}

\thicklines\qbezier(170,20)(190,-10)(210,20)

\qbezier(170,40)(190,70)(210,40)

\qbezier(210,40)(240,-20)(270,40)

\qbezier(210,20)(240,70)(270,20)

\qbezier(110,10)(140,-10)(170,40)

\qbezier(110,40)(130,10)(160,10)

\qbezier(110,20)(130,45)(160,50)

\qbezier(110,50)(140,60)(170,20)

\put(118,28){\circle*{6}}

\put(163,28){\circle*{6}}

\put(216,28){\circle*{6}}

\put(145,45){\circle*{6}}

\put(145,11){\circle*{6}}

\put(265,28){\circle*{6}}

\put(285,18){$\cX$ mod 5}

\put(116,38){1}

\put(175,38){1}

\put(140,50){1}

\put(145,15){1}

\put(216,38){2}

\put(265,38){2}

\qbezier(170,112)(190,82)(210,112)

\qbezier(170,132)(190,162)(210,132)

\qbezier(210,132)(240,72)(270,132)

\qbezier(210,112)(240,162)(270,112)

\qbezier(110,112)(140,162)(170,112)

\qbezier(110,132)(140,72)(170,132)

\put(30,92){$\CM(\Z[\sqrt{-55}])$}

\put(285,92){$\CM(\Z[\frac{1+\sqrt{-55}}{2}])$}

\put(118,120){\circle*{6}}

\put(163,120){\circle*{6}}

\put(216,120){\circle*{6}}

\put(265,120){\circle*{6}}

\put(285,130){$\cX$ mod 11}

\put(116,130){1}

\put(175,130){1}

\put(216,130){1}

\put(265,130){1}

\end{picture}

\quad

\quad

Hence, blowing-down $\cX$ as above, we obtain the thickness of the specialization of each
Weierstrass point $P\in\WP(\cW)$. Applying the rest of the algorithm just as in \S \ref{S-S
exmpl}, we obtain that the model $\cW$ over $\Z[1/2]$ is given by the equation:
\[
\boxed{ y^2=(-3x^4 + 32x^3 -130x^2 + 237x - 163)\cdot(x^4 - 8x^3 + 34x^2 - 83x + 81). }
\]
This curve also coincides with the one conjectured by Kurihara (cf. \cite{Kur}) in this case.

\section{Atkin-Lehner quotients}\label{ALquo}

In \S \ref{Algorithm} we gave an algorithm which in principle works for any hyperelliptic Shimura curve of odd discriminant admitting a Weierstrass model $\cW$ obtained by blowing-down exceptional divisors of $\cX$. However, this algorithm exploits the instruction \emph{IndexFormEquation}, which is implemented in  \emph{MAGMA} only for small degree field extensions. As long as the genus increases, the degrees of the fields involved in the computation become so large that make impossible to proceed with the algorithm.

In this section we shall explain how to adapt the algorithm of \S \ref{Algorithm} to compute equations of hyperelliptic quotients of Shimura curves by Atkin-Lehner involutions. We expect that the degrees of the fields involved in this case will be smaller and, consequently, we shall be able to compute more examples.

\subsection{Quotient of the special fiber}\label{Quotfib}

As above, denote by $\cX/\Z$ Morita's integral model of $X=X_0^D$. Write $Y=X/\langle\omega_m\rangle$ and $\cY=\cX/\langle\omega_m\rangle$. Due to Cerednik-Drinfeld's uniformization, we have an explicit description of the fiber $\cX_p$ at $p\mid D$ and the action of the Atkin-Lehner involutions on its set of irreducible components and singular points.
This allows us to compute the irreducible components of the fiber $\cY_p$. In order to obtain the thicknesses of its singular points $(\cY_p)_{\rm{sing}}$, recall that the completed local ring of any singular point $x$ of $\cX_p$ is of the form:
\[
\widehat{O_{\cX',x'}}\simeq\widehat{O_{S',\mathfrak{p}}}[[u,v]]/(uv-c) \quad c\in \mathfrak{m}_{\mathfrak{p}}.
\]
Here, $u$ and $v$ vanish respectively on each of the irreducible components that meet in $x$.

Let $\pi:\cX\ra\cY$ be the quotient map.
If $\omega_m$ fixes $x$ there are two possibilities: $\omega_m$ fixes $u$ and $v$ or $\omega_m$ exchanges them. If $\omega_m$ fixes $u$ and $v$, the completed local ring of the image $\pi(x)$ is given by
\[
\widehat{O_{\cY',\pi(x')}}\simeq\widehat{O_{S',\mathfrak{p}}}[[x,y]]/(xy-c^2),
\]
where the induced pull-back $\pi^*:\widehat{O_{\cY',\pi(x')}}\ra\widehat{O_{\cX',x'}}$ is given by $x\mapsto u^2,\;y\mapsto v^2$. Thus the thickness of the singular point $\pi(x)$ is twice the thickness of $x$. If $\omega_m u=v$, the completed local ring of the image $\pi(x)$ is given by
\[
\widehat{O_{\cY',\pi(x')}}\simeq\widehat{O_{S',\mathfrak{p}}}[[z]]/(z-c),
\]
where the induced pull-back $\pi^*:\widehat{O_{\cY',\pi(x')}}\ra\widehat{O_{\cX',x'}}$ is given by $z\mapsto uv$. Thus $\pi(x)$ becomes a non-singular point of $\cY_p$. Finally, if $\omega_m(x)=x'\neq x$ the map $\pi$ is not ramified at $x$. Hence it provides an isomorphism of local rings $O_{\cX,x}\simeq O_{\cY,\pi(x)}$. This implies that the thickness of $\pi(x)$ coincides with that of $x$.
Notice that, since we control the singular specialization of Heegner points in $\cX_p$, we also control that of their image in $\cY_p$.

\subsection{Weierstrass points, leading coefficients and fields of definition}

We shall assume that there exists a quadratic order $R_\infty\subset K_\infty$ of discriminant prime-to-$D$ and class number $h_{R_\infty}=1$ such that $\emptyset\neq\CM(R_\infty)\subset X(K_\infty)$. Assume also that $Y$ is hyperelliptic and that the hyperelliptic involution $\omega$ of $Y$ is the image of $\omega_n$ for some $n\mid D$. Notice that all hyperelliptic Shimura curves in Table 1 verify these assumptions. Clearly $n\neq m$ since $\omega_m$ is trivial in $Y$. Finally, assume that blowing-down suitably exceptional divisors of $\cY$ we can obtain a Weierstrass model $\cW_Y$ of $Y$.

As above, the set of Weierstrass points $\WP(Y)$ coincides with the set of fixed points of $\omega$. Let $\pi(P)\in\WP(Y)$. Then $\pi(P)=\omega(\pi(P))=\pi(\omega_n(P))$, thus $\omega_n(P)=P$ or $\omega_n(P)=\omega_m(P)$. It follows that the set $\WP(Y)$ is the image of the union of the set of fixed points of $\omega_n$ and of $\omega_{m}\circ\omega_n=\omega_{n\cdot m/\gcd(m,n)^2}$.
By Theorem \ref{WP}, this set coincides with a set of Heegner points $\bigsqcup_i\CM(R_i)$, where $R_i^0=\Q(\sqrt{-n})$ or $\Q(\sqrt{-n\cdot m})$.

Recall that if $P\in\CM(R_i)$ is fixed by $\omega_D$, then $\Q(P)$ can be computed by means of Theorem \ref{subRCF}. Besides, if $P$ is fixed by $\omega_n$, $n\neq D$, then the field of definition of $P$ is just $H_{R_i}$ by \cite[Theorem 5.12]{GoRo1}.
The following proposition describes the field of definition of each $\pi(P)\in\WP(Y)$:
\begin{proposition}
Let $n\neq m$ be divisors of $D$. Let $P\in\CM(R)$ be a Heegner point fixed by $\omega_n$. Write $Y=X/\langle\omega_m\rangle$ and set $\pi:X\ra Y$ for the quotient map. Fix an embedding $H_R\subset \C$ and let $c$ denote complex conjugation.
\begin{itemize}
\item[(1)] If $m\mid n$ then $\Q(\pi(P))$ is the subfield of $\Q(P)$ fixed by $\Phi_R(\mathfrak{m})$, where $\mathfrak{m}$ is the unique ideal of $R$ of norm $m$.
\item[(2)] If $\omega_m(P)=\omega_D(P)$ (i.e. either $m=D$ or $n=D/m$) then $\Q(\pi(P))$ is the subfield of $\Q(P)$ fixed by $c\cdot\Phi_{R_i}([\mathfrak{a}]$, where $\mathfrak{a}$ is an ideal of $R$ (depending on $P$) satisfying
\begin{equation}\label{eqcond}
B\simeq\left(\frac{-n,\frac{D}{n}\cdot\Norm_{R^0/\Q}(\mathfrak{a})}{\Q}\right).
\end{equation}
\item[(3)] If $\omega_m(P)\neq\omega_D(P)$ and $m\nmid n$ then $\Q(\pi(P))=\Q(P)$.
\end{itemize}
\end{proposition}
\begin{proof}
This follows immediately from the fact that if $m\mid n$ then $\omega_m(P)=P^{\Phi_R(\mathfrak{m})}$ (cf. \cite[Lemma 5.9]{GoRo1}), if $\omega_m(P)=\omega_D(P)$ then $\omega_m(P)=P^{c\cdot\Phi_{R_i}([\mathfrak{a}])}$ (cf. \cite[Lemma 5.10]{GoRo1}), and if neither $m\mid n$ nor $\omega_m(P)=\omega_D(P)$ then $\omega_m$ acts transitively on the $\Gal(H_R/\Q)$-orbit of $P$.
\end{proof}

Just as in Remark \ref{remalgquat}, the ideal $\mathfrak{a}$ depends on $P$ but its class $\{\mathfrak{a}\}\in\Pic(R_i)/\Pic(R_i)^2$ only depends on $R$ and determines the isomorphism class of $\Q(\pi(P))$ for every $P\in\CM(R)$. Furthermore, in our particular setting where $[H:H_R]$ is odd, the class $\{\mathfrak{a}\}$ is uniquely determined by \eqref{eqcond}.

As in the previous case, the model $\cY$ can be non-hyperelliptic (i.e. $\cY/\langle\omega\rangle$ may not be smooth over $\Z$). According to our previous assumptions, we can turn it into an hyperelliptic model $\cW_Y/\Z$ by blowing-down suitably irreducible components. By means of formula \eqref{thickblow_d}, we can recover the thicknesses of the singular points of the fiber $(\cW_Y)_p$. Since we control the specialization of the Weierstrass points $\WP(Y)$ in $\cY_p$, we also control the specialization of the Weierstrass points in $\WP(\cW_Y)$.

Notice that there may exist Weierstrass points $\pi(P)\in\WP(Y)$ specializing to non-singular points on $\cY_p$, but having singular specialization on $(\cW_Y)_p$. This happens because their specialization on $\cY_p$ lie on irreducible components which were blown-down in order to obtain $\cW_Y$. By means of \eqref{redcc}, we control the irreducible component where the specialization $P$ lies. Hence we control the singular specialization of $\pi(P)$ in the fiber $(\cW_Y)_p$.

Choose $P_\infty\in\CM(R_\infty)$. Since $h_{R_\infty }=1$, the set $\CM(R_\infty)$ is a $W(D)$-orbit. Moreover, $\pi(\omega_n(P_\infty))=\omega(\pi(P_\infty))\neq\pi(P_\infty)$ since $\omega_n(P_\infty)\neq\omega_m(P_\infty)$, and $\pi(P_\infty)$ is defined over a subfield of $K_\infty$. This implies that we can set $\pi(P_\infty)$ and $\omega(\pi(P_\infty))$ to be our points at infinity.

Once we fix the points at infinity, the model $\cW_Y$ is defined, over $\Z[1/2]$, by an equation of the form
\[
y^2=R(x)=\prod_i p_{R_i}(x),
\]
where each of the polynomials $p_{R_i}(x)$ is attached to $\pi(\CM(R_i))$, and we control the field that each one defines.

We deduced in \S\ref{hypShi} that the valuation of the leading coefficient $a_R$ at any prime $p$ can be obtained from the intersection index between $\pi(P_\infty)$ and $\omega(\pi(P_\infty))$ at $p$.
By the projection formula,
\begin{equation}\label{projfor}
(\pi(P_\infty),\pi(\omega_n(P_\infty)))_p=(P_\infty,\pi^*\pi(\omega_n(P_\infty)))_p=(P_\infty,\omega_n(P_\infty))_p+(P_\infty,\omega_{n'}(P_\infty))_p,
\end{equation}
where $n'=\frac{n\cdot m}{\gcd(m,n)^2}$.
Hence, the valuation of the leading coefficient at any prime, $$\nu_p(a_R)=\left(1-\left(\frac{K_\infty}{p}\right)\right)(\pi(P_\infty),\pi(\omega_n(P_\infty)))_p,$$
can be computed by means of \eqref{ordcont}.
Since the leading coefficient $a_{R_i}$ of each $p_{R_i}(x)$ also detects whether $P_\infty$ specializes to the same supersingular point as an element of $\CM(R_i)$, we can compute each $a_{R_i}$ just as in \S\ref{Algorithm}.

At this point, assuming that $D$ is odd, we can proceed with the algorithm of \S\ref{Algorithm} in order to obtain an equation for $\cW_Y$. Indeed, we control the leading coefficient of each $p_{R_i}(x)$, their splitting field and the singular specialization of any $\pi(P)\in\WP(\cW_Y)$.

\subsection{Example}

Let $X=X_0^{35}/\Q$ be the Shimura curve of discriminant $35$. In this section we shall compute the quotient curve $Y=X/\langle\omega_5\rangle$.
Since $X$ is itself hyperelliptic we deduce that $Y$ is hyperelliptic. Moreover, we check that it satisfies the assumptions of the previous section.

Write $\pi:\cX\ra \cY$ for the quotient map as above. The set of Weierstrass points of $Y$ is the image through $\pi$ of the set of Heegner points $\mathfrak{S}=\CM(R^{35})\sqcup\CM(R^{35}_0)\sqcup\CM(R^{7})\sqcup\CM(R_0^{7})$, where $R^{35}=\Z[\frac{1+\sqrt{-35}}{2}]$, $R^{35}_0=\Z[\sqrt{-35}]$, $R^{7}=\Z[\frac{1+\sqrt{-7}}{2}]$ and $R^{7}_0=\Z[\sqrt{-7}]$. We obtain that $R^{35}$ has Picard number $2$, $R_0^{35}$ has Picard number $6$, and both $R^{7}$ and $R_0^{7}$ have Picard number 1. Here, we present a diagram that describes the special fibers of $\cX$ at $p=5,7$ and the specialization of $\mathfrak{S}$ computed using the techniques of \S\ref{Algorithm}.

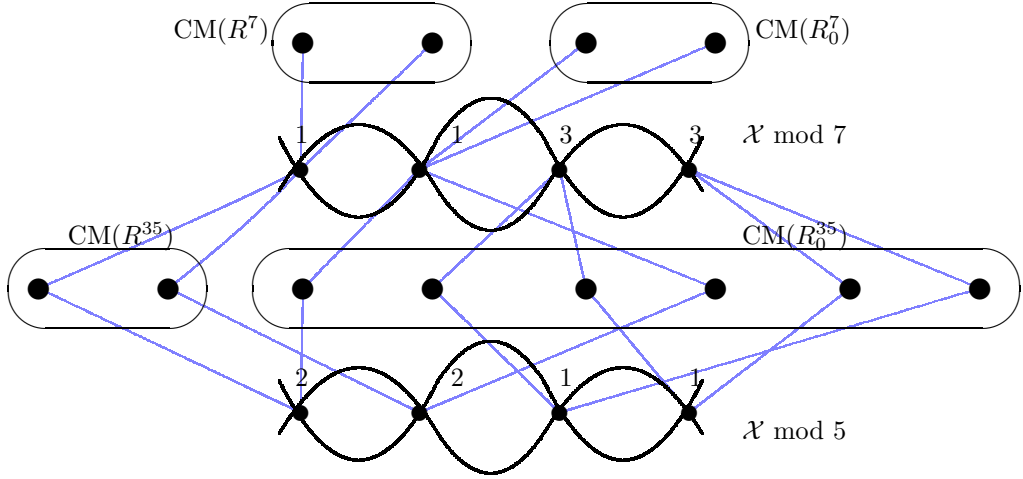
\begin{picture}(300,200)

\color[rgb]{0.5,0.5,1}

\qbezier(19,75)(100,110)(118,120)
\qbezier(68, 75)(100,100)(118,120)
\qbezier(119, 75)(163,120)(163,120)
\qbezier(226, 75)(216,120)(216,120)
\qbezier(168, 75)(216,120)(216,120)
\qbezier(275, 75)(163,120)(163,120)
\qbezier(326,75)(265,120)(265,120)
\qbezier(375, 75)(265,120)(265,120)

\qbezier(119,168)(118,120)(118,120)
\qbezier(168,168)(118,120)(118,120)
\qbezier(226,168)(163,120)(163,120)
\qbezier(275,168)(163,120)(163,120)

\qbezier(19,75)(118,28)(118,28)
\qbezier(68, 75)(163,28)(163,28)
\qbezier(119, 75)(118,28)(118,28)
\qbezier(168, 75)(216,28)(216,28)
\qbezier(226,75)(265,28)(265,28)
\qbezier(275, 75)(163,28)(163,28)
\qbezier(326,75)(265,28)(265,28)
\qbezier(375, 75)(216,28)(216,28)

\color{black}

\put(19, 75){\circle*{8}}
\put(68, 75){\circle*{8}}

\put(119,75){\circle*{8}}
\put(168,75){\circle*{8}}
\put(226,75){\circle*{8}}
\put(275,75){\circle*{8}}
\put(326,75){\circle*{8}}
\put(375,75){\circle*{8}}

\put(119,168){\circle*{8}}
\put(168,168){\circle*{8}}
\put(226,168){\circle*{8}}
\put(275,168){\circle*{8}}

\put(45, 75){\oval(75, 30)}
\put(245, 75){\oval(290, 30)}

\put(145, 168){\oval(75, 30)}
\put(250, 168){\oval(75, 30)}

\thicklines

\qbezier(170,20)(190,-10)(210,20)

\qbezier(170,40)(190,70)(210,40)

\qbezier(210,40)(240,-20)(270,40)

\qbezier(210,20)(240,70)(270,20)

\qbezier(110,40)(140,-20)(170,40)

\qbezier(110,20)(140,70)(170,20)

\put(118,28){\circle*{6}}

\put(163,28){\circle*{6}}

\put(216,28){\circle*{6}}

\put(265,28){\circle*{6}}

\put(285,18){$\cX$ mod 5}

\put(116,38){2}

\put(175,38){2}

\put(216,38){1}

\put(265,38){1}

\qbezier(170,112)(190,82)(210,112)

\qbezier(170,132)(190,162)(210,132)

\qbezier(210,132)(240,72)(270,132)

\qbezier(210,112)(240,162)(270,112)

\qbezier(110,112)(140,162)(170,112)

\qbezier(110,132)(140,72)(170,132)

\put(30,92){$\CM(R^{35})$}

\put(285,92){$\CM(R^{35}_0)$}

\put(70,170){$\CM(R^{7})$}

\put(290,170){$\CM(R^{7}_0)$}

\put(118,120){\circle*{6}}

\put(163,120){\circle*{6}}

\put(216,120){\circle*{6}}

\put(265,120){\circle*{6}}

\put(285,130){$\cX$ mod 7}

\put(116,130){1}

\put(175,130){1}

\put(216,130){3}

\put(265,130){3}

\end{picture}

\quad

By Cerednik-Drinfeld's description of the fiber $\cX_5$, we know that $\omega_5$ exchanges its irreducible components, moreover, it exchanges its singular points of thickness 1 and its singular points of thickness 2. Similarly, $\omega_5$ fixes the irreducible components of $\cX_7$, exchanges its singular points of thickness 3 and fixes its singular points of thickness 1.
Applying the recipe detailed in \S\ref{Quotfib}, we obtained that the specialization of $\pi(\mathfrak{S})$ and the special fibers $\cY_5$ and $\cY_7$ are given by the following diagram:

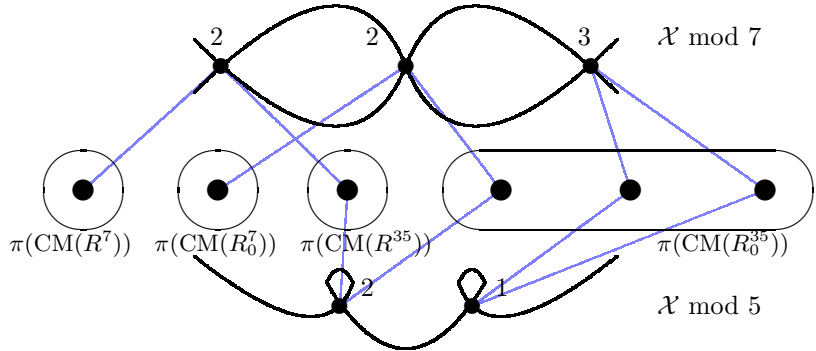
\begin{picture}(300,170)

\color[rgb]{0.5,0.5,1}

\qbezier(68, 75)(120,122)(120,122)
\qbezier(119, 75)(190,122)(190,122)
\qbezier(168, 75)(120,122)(120,122)
\qbezier(226, 75)(190,122)(190,122)
\qbezier(275, 75)(260,122)(260,122)
\qbezier(326, 75)(260,122)(260,122)

\qbezier(168, 75)(165,31)(165,31)
\qbezier(226, 75)(165,31)(165,31)
\qbezier(275, 75)(215,31)(215,31)
\qbezier(326, 75)(215,31)(215,31)

\color{black}

\put(68,75){\circle*{8}}

\put(119,75){\circle*{8}}

\put(168,75){\circle*{8}}

\put(226,75){\circle*{8}}

\put(275,75){\circle*{8}}

\put(326,75){\circle*{8}}

\put(68, 75){\oval(30, 30)}

\put(119, 75){\oval(30, 30)}

\put(168, 75){\oval(30, 30)}

\put(274, 75){\oval(140, 30)}

\thicklines

\qbezier(210,40)(220,10)(270,50)
\qbezier(160,40)(190,-10)(220,40)
\qbezier(160,40)(165,50)(170,40)
\qbezier(110,50)(160,10)(170,40)
\qbezier(210,40)(215,50)(220,40)

\put(165,31){\circle*{6}}

\put(215,31){\circle*{6}}

\put(285,28){$\cX$ mod 5}

\put(173,35){2}

\put(224,35){1}

\qbezier(110,112)(170,172)(190,122)
\qbezier(110,132)(170,72)(190,122)
\qbezier(190,122)(210,72)(270,132)
\qbezier(190,122)(210,172)(270,112)

\put(40,52){\small$\pi(\CM(R^{7}))$}

\put(95,52){\small$\pi(\CM(R^{7}_0))$}

\put(150,52){\small$\pi(\CM(R^{35}))$}

\put(285,52){\small$\pi(\CM(R^{35}_0))$}

\put(120,122){\circle*{6}}
\put(190,122){\circle*{6}}
\put(260,122){\circle*{6}}

\put(285,130){$\cX$ mod 7}

\put(116,130){2}
\put(175,130){2}
\put(255,130){3}

\end{picture}

\quad

Let $R_\infty$ be the maximal order of $K_\infty=\Q(\sqrt{-43})$. Since $h_{R_\infty}=1$ and $\CM(R_\infty)\neq\emptyset$,
we choose $\{\pi(P_\infty),\pi(\omega_D(P_\infty))\}\subseteq\pi(\CM(R_\infty))$ to be our points at infinity. Then, by \eqref{projfor},
\[
\nu_p(a_R)=\left(1-\left(\frac{K_\infty}{p}\right)\right)((P_\infty,\omega_D(P_\infty))_p+(P_\infty,\omega_{D/m}(P_\infty))_p)
\]
In order to compute $(P_\infty,\omega_D(P_\infty))_p$, we apply \eqref{dioph} and it follows that
\[
43\cdot 35=n+\Norm(\lambda^-)\cdot 35\cdot p,\;\;\quad n=43\cdot\Norm(\lambda_+)\in\Z.
\]
Since $35\mid n$ and 7,5 are inert in $K_\infty$, we deduce that $35^2\cdot n'=n$. Thus,
\[
43=35\cdot n'+\Norm(\lambda^-)\cdot p,\;\;\quad n'\in\Z.
\]
Hence the solutions are $n'=0$, $p=43$, $\Norm(\lambda^-)=1$ and $n'=1$, $p=2$, $\Norm(\lambda^-)=4$. Applying formula \eqref{ordcont}, we deduce that $(P_\infty,\omega_D(P_\infty))_{43}=1$ and $(P_\infty,\omega_D(P_\infty))_{2}=2$.

Similarly for $(P_\infty,\omega_{D/m}(P_\infty))_p$, we apply formula \eqref{dioph} obtaining:
\[
43\cdot 7=n+\Norm(\lambda^-)\cdot 35\cdot p,\;\;\quad n=43\cdot\Norm(\lambda_+)\in\Z.
\]
As above, $7\mid n$ and it is inert in $K_\infty$, hence $7^2\cdot n'=n$ and it follows that
\[
43=7\cdot m'+5\cdot\Norm(\lambda^-)\cdot p,\;\;\quad m'\in\Z.
\]
This implies $m'\equiv 4\;(\rm{mod}\;5)$ and, thus, $m'=4$, $p=3$, $\Norm(\lambda^-)=1$. By means of \eqref{ordcont} we have that $(P_\infty,\omega_D(P_\infty))_{3}=1$.

Therefore the unique primes that divide $a_R$ are $43$, $3$ and $2$ and their valuations are $\nu_{43}(a_R)=1$, $\nu_{3}(a_R)=2$ and $\nu_{2}(a_R)=4$. Moreover, we can compute the specialization of $P_\infty$ and $\omega(P_\infty)$ at $p=3,43$ and determine which Weierstrass point lie at the same supersingular point as them. We obtained that $\nu_{43}(a_{R_0^{35}})=1$ and $\nu_3(a_{R_0^{7}})=2$. We can not control the 2-valuation of any leading coefficient $a_{R_i}$ but we know the valuation of the product $4=\nu_2(a_R)=\sum_i\nu_2(a_{R_i})$ and this gives an upper bound for all of them.

Finally, applying the rest of the algorithm of \S\ref{Algorithm}, we obtained that $Y$ is defined by the equation:
\[
\boxed{ y^2=-x\cdot (9x+4)\cdot(4x+1)\cdot(172x^3 + 176x^2 + 60x + 7). }
\]

\subsection{Results}\label{results}

In this section we present a table with all the equations obtained using the algorithms explained in \S \ref{Algorithm} and \S \ref{ALquo}:
$$
\begin{array}{c|c|c}
g & \mbox{\emph{curve}} & y^2=p(x)\\ \hline
3 & X_0^{39} & y^2=-(7x^4 + 79x^3 + 311x^2 + 497x + 277)\cdot(x^4 + 9x^3 + 29x^2 + 39x + 19)\\[2pt]
3 & X_0^{55} & y^2=-(3x^4 - 32x^3 +130x^2 - 237x + 163)\cdot(x^4 - 8x^3 + 34x^2 - 83x + 81)\\[2pt]
2 & X_0^{35}/\langle\omega_5\rangle & y^2=-x\cdot (9x+4)\cdot(4x+1)\cdot(172x^3 + 176x^2 + 60x + 7)\\[2pt]
2 & X_0^{51}/\langle\omega_{17}\rangle & y^2=-x\cdot(7x^3 + 52x^2 + 116x + 68)\cdot(x-1)\cdot(x+3)\\[2pt]
2 & X_0^{57}/\langle\omega_{3}\rangle & y^2=-(x-9)\cdot(x^3 - 19x^2 + 119x - 249)\cdot(7x^2 - 104x + 388)\\[2pt]
2 & X_0^{65}/\langle\omega_{13}\rangle & y^2=-(x^2 - 3x + 1)\cdot(7x^4 - 3x^3 - 32x^2 + 25x - 5)\\[2pt]
2 & X_0^{65}/\langle\omega_{5}\rangle & y^2=-(x^2 + 7x + 9)\cdot(7x^4 + 81x^3 + 319x^2 + 508x + 268)\\[2pt]
2 & X_0^{69}/\langle\omega_{23}\rangle & y^2=-x\cdot(x+4)\cdot(4x^4 - 16x^3 + 11x^2 + 10x + 3)\\[2pt]
2 & X_0^{85}/\langle\omega_{5}\rangle & y^2=-(3x^2 - 41x + 133)\cdot(x^4 - 23x^3 + 183x^2 - 556x + 412)\\[2pt]
2 & X_0^{85}/\langle\omega_{85}\rangle & y^2=(x^2 - 3x + 1)\cdot (x^4 + x^3 - 15x^2 + 20x - 8)\\[2pt]
\end{array}
$$
\begin{centerline}{\bf Table 2}\end{centerline}

\bibliographystyle{plain}

\bibliography{santi}

\end{document}